\documentclass{article}
\usepackage{harvard}
\usepackage{graphicx}
\usepackage{camnum}
\usepackage{amsmath}
\usepackage{enumerate}
\usepackage{algorithm}
\usepackage{algpseudocode}
\usepackage{tikz,pgfplots}
\usepackage{tikz-cd}

\font\smallcurly=rsfs10 scaled 800

\allowdisplaybreaks

\newcommand{\hyper}[5]
       {{}_{#1} F_{#2} \!\left[
           \begin{array}{l}
         #3;\\#4;
           \end{array}#5\right]
       } 
\newcommand{\ee}{\mathrm{e}}
\newcommand{\ii}{{\mathrm i}}
\newcommand{\DDD}{\DD{D}}
\newcommand{\EEE}{\DD{E}}

\newcommand{\PPP}{\DD{P}}
\newcommand{\PP}{\CC{P}}
\newcommand{\LL}{\CC{L}}
\newcommand{\G}{\CC{\Gamma}}
\newcommand{\HHH}[1]{{\CC{H}^{\hspace*{-6pt}\raisebox{4pt}{\footnotesize$\circ$}}_2}^{\hspace*{-4pt}\raisebox{-3pt}{$\scriptstyle #1$}}}

\newcommand{\lA}{\nearrow}
\newcommand{\rA}{\searrow}

\pgfplotsset{compat=1.12}

\title{Orthogonal systems for time-dependent spectral methods}
\author{Arieh Iserles\\Department of Applied Mathematics and Theoretical Physics\\University of Cambridge\\Centre for Mathematical Sciences\\Wilberforce Rd, Cambridge CB3 0WA\\United Kingdom\\[3pt]email: \texttt{ai@maths.cam.ac.uk}}

\begin{document}
\maketitle
\tableofcontents

\begin{abstract}
  This paper is concerned with orthonormal systems in real intervals, given with zero Dirichlet boundary conditions. More specifically, our interest is in systems with a skew-symmetric differentiation matrix (this excludes  orthonormal polynomials). We consider a simple construction of such systems and pursue its ramifications. In general, given any $\CC{C}^1(a,b)$ weight function such that $w(a)=w(b)=0$, we can generate an orthonormal system with a skew-symmetric differentiation matrix. Except for the case $a=-\infty$, $b=+\infty$, only a limited number of powers of that matrix is bounded and we establish a connection between properties of the weight function and boundedness. In particular, we examine in detail two weight functions: the Laguerre weight function $x^\alpha \ee^{-x}$ for $x>0$ and $\alpha>0$ and the ultraspherical weight function $(1-x^2)^\alpha$, $x\in(-1,1)$, $\alpha>0$, and establish their properties. Both weights share a most welcome feature of {\em separability,\/} which allows for fast computation. The quality of approximation is highly sensitive to the choice of $\alpha$ and we discuss how to choose optimally this parameter, depending on the number of zero boundary conditions.
\end{abstract}


\vspace{5pt}

\section{Introduction and motivation}

This work is motivated by spectral methods for time-dependent partial differential equations (PDEs) of the form
\begin{equation}
  \label{eq:1.1}
  \frac{\partial u}{\partial t}=\mathcal{L}u+f(x,u),\qquad t\geq0,\quad x\in\Omega,
\end{equation}
where $\mathcal{L}$ is a well-posed linear operator and $\Omega\subseteq\BB{R}^d$, given with an initial condition for $u(x,0)$ and appropriate boundary conditions on $\partial\Omega$. Standard examples are $\mathcal{L}=\Delta$ with $f\equiv0$ (the {\em diffusion\/} equation) or $f$ a cubic polynomial in $u$ with real zeros (the {\em Fitzhugh--Nagumo\/} equation) and $\mathcal{L}=\ii \Delta$ with either $f=-\ii V(x)$ (the {\em linear Schr\"odinger\/} equation) or $f=-\ii\lambda|u|^2u$ (the {\em nonlinear Schr\"odinger\/} equation with standard cubic nonlinearity). 

In this paper we are concerned by spectral methods applied in tandem with a splitting approach. As an example, we commence by approximating locally the solution of \R{eq:1.1} using the Strang splitting,
\begin{equation}
  \label{eq:1.2}
  u^{n+1}=\ee^{\frac12\Delta t \mathcal{L}} \ee^{\Delta t f} \ee^{\frac12\Delta t \mathcal{L}} u^n,\qquad n\geq0,
\end{equation}
where $u^n(x)$ is an approximation to $u(x,n\Delta t)$. Here $\ee^{t\mathcal{L}}v$ is a shorthand for a numerical solution at $t$ of $\partial u/\partial t=\mathcal{L}u$, $u(0)=v$, while $\ee^{tf}v$ denotes a numerical solution of the ordinary differential equation (ODE) $\D u/\D t=f(x,u)$, $u(0)=v$. The splitting \R{eq:1.2}, which incurs a local error of $\O{(\Delta t)^3}$, is but one example of {\em operatorial splittings\/} \cite{bader14eas,blanes20hoe,mclachlan02sm} and is intended here to illustrate a general point, namely that the solution of `complicated' PDEs can be reduced to the solution of `simple' PDEs and  ODEs. Once done correctly, this procedure is consistent with  eventual quality of the solution,  concerning accuracy and stability alike.

Another benefit of \R{eq:1.2} and of similar splittings is that it is consistent with conservation of the $\CC{L}_2$ energy. Many {\em dispersive\/} equations, e.g.\ Schr\"odinger (linear or nonlinear), Gross--Pitaevskii, Dirac, Klein--Gordon and Korteveg--De Vries, conserve the $\CC{L}_2$ norm of the solution. This  often represents a highly significant physical feature and it is vital to respect it under discretisation. (Note that conservation of $\CC{L}_2$ norm automatically implies numerical stability.) Because of the special form of \R{eq:1.2} (and of similar splittings) the overall numerical scheme preserves $\CC{L}_2$ energy if both the discretisations of $\ee^{ t\mathcal{L}}$ and $\ee^{tf}$ do so. Insofar as $\ee^{tf}$ is concerned, we can use the very extensive and robust existing theory \cite{hairer06gni}, e.g.\ use a symplectic method (which automatically also preserves the $\CC{L}_2$ norm). It is more challenging to ensure that $\|\ee^{t\mathcal{L}}v\|_2=\|v\|_2$ for every $v$, in other words that the discretisation of $\ee^{t\mathcal{L}}$ is {\em unitary.\/} 

In this paper we are concerned with spectral methods for time-dependent problems \cite{canuto06sm,hesthaven07std,trefethen00smm}. In a nutshell, we commence from a set $\Phi=\{\varphi_n\}_{n\in\bb{Z}_+}$, where each $\varphi_n$ is defined in $\Omega$ and endowed with appropriate regularity, which are  orthonormal in the standard $\CC{L}_2$ inner product,
\begin{displaymath}
  \int_\Omega \varphi_m(x)\varphi_n(x) \D x=\delta_{m,n},\qquad m,n\in\BB{Z}_+,
\end{displaymath}
and complete in $\CC{L}_2(\Omega)$, and expand a solution in the basis $\Phi$,
\begin{displaymath}
  u(x,t)=\sum_{n=0}^\infty u_n(t) \varphi_n(x),
\end{displaymath}
where the $u_n(0)$s are determined by expanding the initial condition $u_0$, while the $u_n(t)$s, $t>0$ are typically evolved by  Galerkin conditions, which for \R{eq:1.1} read
\begin{displaymath}
  u_m'(t)=\sum_{n=0}^\infty u_n(t) \langle \mathcal{L}\varphi_n+f(\,\cdot\,,\varphi_n),\varphi_m\rangle,\qquad m\in\BB{Z}_+.
\end{displaymath}
In a practical method we truncate the expansion and the range of $m$, thereby obtaining a finite-dimensional linear system of ODEs. 

Substantive advantage of spectral methods is that expansions in orthonormal bases typically converge very rapidly indeed: for example, orthogonal polynomials converge in a finite interval to analytic (in the interval and its neighbourhood) functions at an exponential rate. Therefore, the number of degrees of freedom, compared to the more usual finite difference or finite element methods, is substantially small. While this is not the entire truth -- finite differences and finite elements produce sparse linear algebraic systems while spectral elements yield dense matrices which sometimes can be also ill conditioned and , moreover, expanding a function in an orthonormal basis can be potentially costly -- spectral methods are often the approach of choice in numerical computations.  In the specific context of time-dependent problems, however, naive spectral methods are unstable \cite{hesthaven07std}. This motivates us to consider the major concept of a {\em differentiation matrix.\/}

In the sequel we restrict our narrative to the univariate case, $\Omega\subseteq\BB{R}$, for a number of reasons. Firstly, surprisingly, even the univariate case  (as we hope to persuade the reader) is dramatically incomplete. Secondly, it lays the foundations to a multivariate case, whether by tensorial extension to parallelepipeds or by more advanced means which we intend to explore in a future paper. 

The set $\Phi$ being a basis of $\CC{L}_2(\Omega)\cap\CC{C}^1(\Omega)$, any function therein can be expressed as a linear combination of the $\varphi_n$s, and this is particularly true with regards to the derivatives $\varphi_m'$. This yields a linear map represented by the infinite-dimensional matrix $\DDD$ such that
\begin{displaymath}
  \DDD_{m,n}=\int_\Omega \varphi_m'(x) \varphi_n(x)\D x,\qquad m,n\in\BB{Z}_+.
\end{displaymath}

It is very simple to prove by integration by parts that the differentiation operator $\MM{D}=\partial/\partial x$ is skew Hermitian in the following three configurations of boundary conditions:
\begin{enumerate}
\item The {\em torus} $\Omega=\BB{T}$ (i.e.\ periodic boundary conditions); 
\item The {\em Cauchy problem:\/} $\Omega=\BB{R}$; and
\item Zero Dirichlet boundary conditions on the boundary of $\Omega\subset\BB{R}$.
\end{enumerate}
In that case $\|\ee^{t\Mm{D}}\|=1$ and, $\MM{D}^2$ being Hermitian and negative semidefinite, $\|\ee^{t\Mm{D}^2}\|\leq1$. More generally, $\MM{D}^{2\ell+1}$ is skew Hermitian and $(-1)^{\ell-1} \MM{D}^{2\ell}$ negative semidefinite for all $\ell\in\BB{Z}_+$. Consequently, once $\mathcal{L}=\sum_{\ell=1}^M a_\ell \MM{D}^{\ell}$, where $(-1)^{\ell-1} a_{2\ell}\geq0$, it is trivial to prove that  $\langle \mathcal{L}u,u\rangle\leq0$ for every $u$ in the underlying Hilbert space. 

This feature is retained by a spectral method, provided that $\DDD$ is skew Hermitian, as is its $(N+1)\times(N+1)$ section $\DDD_N$. In the context of the PDE \R{eq:1.1}, it thus follows that, letting  $\mathcal{L}_N=\sum_{\ell=1}^M a_\ell \DDD_N^{\,\ell}$, we have $w^*\mathcal{L}_N w\leq0$ for all $w\in\BB{C}^{M+1}$. It follows, for example, that the $\CC{L}_2$ energy is conserved for $\mathcal{L}=\MM{D}$ (the Schr\"odinger case) and it dissipates for $\mathcal{L}=\MM{D}^2$ (the diffusion equation case). In both cases numerical stability comes in the wash.

The obvious choice of an orthonormal system is a set of orthogonal polynomials -- unless we use (possibly shifted) Legendre polynomials, this means replacing the $\CC{L}_2$ inner product by another, defined by the orthogonality weight function -- but it is clear that this produces a lower-triangular $\DDD$. While there are nontrivial way round it \cite{olver20fau}, there is strong motivation to consider alternative orthonormal systems. 

The periodic case -- without loss of generality, $\Omega=[-\pi,\pi]$ with periodic boundary conditions -- is obvious: we let $\Phi=\{\ee^{\ii nx}\}_{n\in\bb{Z}}$, the Fourier basis. An added bonus is fast expansion by means of a Fast Fourier Transform of any $\CC{L}_2[-\pi,\pi]\cap\CC{C}_{\CC{per}}[-\pi,\pi]$ function in the underlying basis. This is the paradigmatic case whereby a spectral method has few competitors. 

The Cauchy case $\Omega=(-\infty,\infty)$ has been a subject for an extensive recent study \cite{iserles19oss,iserles20for,iserles21daf,iserles21fco}. In particular, all orthonormal systems $\Phi$ such that $\DDD$ is skew Hermitian {\em and\/} tridiagonal have been completely characterised. Specifically, they are in a one-to-one relationship with Borel measures, supported on the entire real line. Let $\D\mu(x)=w(x)\D x$ be such a measure ($w$ might be a generalised function) and $\PPP=\{p_n\}_{n\in\bb{Z}_+}$ the underlying set of orthonormal polynomials. It is elementary that $\PPP$ obeys a three-term recurrence relation
\begin{equation}
  \label{eq:1.3}
  \beta_n p_{n+1}(x)=(x+\alpha_n)p_n(x)-\beta_{n-1}p_{n-1}(x),\qquad n\in\BB{Z}_+,
\end{equation}
where $\beta_{-1}=0$, $\alpha_n\in\BB{R}$ and $\beta_n>0$ for $n\in\BB{Z}_+$. Inverse Fourier transforming $\{w^{1/2} p_n\}_{n\in\bb{Z}_+}$ and multiplying the $n$ term by $\ii^n$, we obtain an orthonormal set $\Psi$, dense in $\CC{L}_2(\BB{R})$ and such that
\begin{displaymath}
  \psi_n'=-\beta_{n-1}\psi_{n-1}+\ii\alpha_n \psi_n+\beta_n\psi_{n+1},\qquad n\in\BB{Z}_+.
\end{displaymath}
Therefore $\DDD$ is skew Hermitian (skew symmetric if $\alpha_n\equiv0$, which is the case once $w$ is an even function) and tridiagonal, while  orthonormality follows by the Plancherel theorem. Tridiagonality is a valuable feature because it is easy to manipulate $\DDD$ (e.g.\ multiply $\DDD_N$ by a vector or approximate $\exp(t\DDD_N)$) and the powers of the infinite-dimensional matrix $\DDD$ (approximating higher derivatives) remain bounded.

This leaves us with the third -- and most difficult -- case, namely zero Dirichlet boundary conditions.\footnote{It is elementary to reduce nonzero Dirichlet conditions to zero ones by reformulating the PDE for $u^{\CC{new}}=u-\rho$, where $\rho$ is any sufficiently regular function that obeys the right Dirichlet boundary conditions on the boundary.} This is the subject of this paper.

A natural inclination is to extend the Fourier-transform-based theory from $(-\infty,\infty)$ to, say, $(-1,1)$. This can be done in one of two obvious ways and, unfortunately, both fail. The first is to choose a measure $\D\mu$ supported by $(-1,1)$, but this leads again to $\Phi$ supported on the entire real line, the only difference being that in this case the closure of $\Phi$ is not $\CC{L}_2(-\infty,\infty)$ but a Paley--Wiener space \cite{iserles19oss}. Another possibility is to abandon altogether the Fourier route and alternatively commence by specifying a $\varphi_0$,  subsequently determining the $\varphi_n$s for $n\in\BB{N}$ and the matrix $\DDD$ consistently with both orthogonality and tridiagonality. A forthcoming paper demonstrates how to do this algorithmically. However, given $\varphi'=\DDD\,\varphi$, it follows by induction that $\varphi^{(s)}=\DDD^{\,s}\varphi$ for all $s\in\BB{Z}_+$. Consistency with zero Dirichlet boundary conditions, though, requires $\varphi_n(\pm1)\equiv0$, and this implies that $\varphi_n^{(s)}(\pm1)\equiv0$ for all $n,s\in\BB{Z}_+$. If $\varphi_0$ is analytic in $(-1,1)$, this means that it necessarily must have an essential singularity at the endpoints. Intuitively, this is bad news, and this is confirmed by numerical experiments that indicate that the $\varphi_n$s develop boundary layers and wild oscillations near $\pm1$ and their approximation power is nil. 

Both ideas above fall short and the current paper embarks on an altogether different approach, abandoning tridiagonality and the Fourier route altogether. Note that the existence of an essential singularity at the endpoints hinged on the fact that all powers of the infinite matrix $\DDD$ are bounded. This is obvious once $\DDD$ is tridigonal (or, more generally, bounded), hence our main idea is to choose an orthonormal set $\DDD$ such that $\DDD^{\,s+1}$ blows up for some $s\in\BB{N}$. At the same time, we wish to retain a major blessing of tridiagonality, namely that $\DDD_N w$ can be computed in $\O{N}$ operations for any $w\in\BB{C}^{N+1}$.  

The main idea underlying this paper is exceedingly simple: given a measure $\D\mu=w\D x$, where $w\in\BB{C}^1(a,b)$, and an underlying set $\PPP=\{p_n\}_{n\in\bb{Z}_+}$ of orthonormal polynomials, we set
\begin{equation}
  \label{eq:1.4}
  \varphi_n(x)=\sqrt{w(x)}p_n(x),\qquad x\in(a,b).
\end{equation}
It follows at once that $\Phi$ is orthonormal with respect to $\CC{L}_2(a,b)$ and it is easy to determine conditions so that $\varphi_n(a)=\varphi_n(b)=0$ for all $n\in\BB{Z}_+$. It is not difficult to specify the conditions on $w$ for skew symmetry of $\DDD$. However, the narrative becomes more complicated once we seek a system such that $\DDD^{\,k}$ is bounded for $k=1,\ldots,s$ and blows up for $k=s+1$. Likewise, it is considerably more challenging to identify systems $\Phi$ that allow for fast computation of $\DDD_N w$ for $w\in\BB{C}^{N+1}$.  Note in passing that $\varphi_n=(-\ii)^n \hat{\psi}_n$: up to rescaling, $\Phi$ consists of Fourier transforms of $\Psi$: we will further elaborate this point in Subsection~4.3.

In Section 2 we introduce the functions \R{eq:1.4} in a more rigorous setting of Sobolev spaces and explore general properties of their differentiation matrices. Section~3 is devoted to two families of weight functions, namely the {\em Laguerre family\/} $w_\alpha(x)=x^\alpha \ee^{-x}\chi_{(0,\infty)}(x)$ and the {\em ultraspherical family\/} $w_\alpha(x)=(1-x^2)^\alpha \chi_{(-1,1)}(x)$. We prove that both families have a {\em separable differentiation matrix.\/} This feature (put to a good use in Section~4) is very special -- indeed, there are good reasons to conjecture that these two families are the only weights with this feature. We present a detailed example of two orthogonal families, {\em generalised Hermite\/} weights $w_\mu(x)=|x|^{2\mu}\ee^{-x^2}$ and {\em Konoplev weights\/} $w_{\alpha,\beta}(x)=|x|^{2\beta+1}(1-x^2)^\alpha \chi_{(-1,1)}(x)$, and prove that their differentiation matrices cannot be separable unless (for Konoplev weights) $\beta=-\frac12$ and the weight reduces to ultraspherical. Finally, in Section~4 we demonstrate how separability of the differentiation matrix can be utilised for fast multiplication of $\DDD_N w$, $w\in\BB{R}^{N+1}$, in $\O{N}$ operations.

The original idea to consider functions of the form \R{eq:1.4} in the specific case of {\em Freud weights\/} has been considered first by \citeasnoun{luong23apw}, who demonstrated that in this specific case $\DDD$ is a skew-symmetric, banded matrix with bandwidth seven. This was a serendipiteous choice: in Section~2 we prove that the only weights that produce a banded matrix in the setting of \R{eq:1.4} are generalised Freud weights!

\setcounter{equation}{0}
\setcounter{figure}{0}
\section{W-functions}

\subsection{The definition and few of its consequences}

Let $(a,b)$ be a non-empty real interval, $-\infty\leq a<b\leq \infty$, and $s\in\BB{N}\cup\{\infty\}$. We denote by $\HHH{s}[a,b]$ the Sobolev space of $\CC{H}_2^s[a,b]$ functions $f$ such that
\begin{displaymath}
  f^{(k)}(a)=f^{(k)}(b)=0,\qquad k=0,\ldots,s-1,
\end{displaymath}
(note that $\CC{C}^{s-1}[a,b]\subset \CC{H}_2^s[a,b]$, therefore the derivatives are well defined) equipped with the inner product
\begin{displaymath}
  \langle f,g\rangle_s=\sum_{k=0}^{s} \int_a^b f^{(k)}(x) g^{(k)}(x)\D x.
\end{displaymath}

A {\em weight function\/} $w\in\CC{L}_2(a,b)\cap\CC{C}^1(a,b)$ is a positive function with all its moments
\begin{displaymath}
  \mu_k=\int_a^b x^k w(x)\D x,\qquad k\in\BB{Z}_+,
\end{displaymath}
bounded. Given a weight functions, we can define (e.g.\ using a Gram--Schmidt process) a set of orthonormal polynomials $\PPP=\{p_n\}_{n\in\bb{Z}_+}$ such that
\begin{equation}
  \label{eq:2.1}
  \int_a^b p_m(x)p_n(x) w(x)\D x=\delta_{m,n},\qquad m,n\in\BB{Z}_+.
\end{equation}
Such a set is unique once we require, for example, that the coefficient of $x^n$ in $p_n$ is always positive. We say that $\varphi_n$ is the $n$th {\em W-function\/} once
\begin{displaymath}
  \varphi_n(x)=\sqrt{w(x)} p_n(x),\qquad n\in\BB{Z}_+,
\end{displaymath}
and let $\Phi=\{\varphi_n\}_{n\in\bb{Z}_+}$. It follows at once from \R{eq:2.1} that $\Phi$ is an orthonormal set with respect to the standard $\CC{L}_2$ inner product. 

\vspace{6pt}
\noindent{\bf Remark 1} The functions $\varphi_n$ inherit some features of orthonormal polynomials, in particular they obey the same three-term recurrence relation. However, the expansion coefficients of an arbitrary function are different:
\begin{Eqnarray*}
  f&\sim& \sum_{n=0}^\infty \hat{f}_n^P p_n,\quad \hat{f}_n^P=\int_a^b f(x) p_n(x) w(x)\D x,\qquad f\in\CC{L}_2((a,b),w\D x),\\
  f&\sim&\sum_{n=0}^\infty \hat{f}_n^\Phi \varphi_n,\quad \hat{f}_n^\Phi=\int_a^b f(x) p_n(x)\sqrt{w(x)}\D x,\qquad f\in\CC{L}_2(a,b).
\end{Eqnarray*}
Moreover, while convergence theory of orthogonal polynomials is well understood, at any rate in compact intervals, the condition for convergence of the $\hat{f}_n^\Phi$s are currently a matter for further research.

\vspace{6pt}
\noindent{\bf Remark 2} An important difference between $\PPP$ and $\Phi$ is that, while the $p_n$ are polynomials, hence analytic functions, the W-functions carry over potential singularities of the weight function. For example, for the Chebyshev weight function $w(x)=(1-x^2)^{-1/2}\chi_{(-1,1)}(x)$ the $\varphi_n$s have weak singularity at the endpoints $\pm1$, while their derivatives possess strong singularity there.

\vspace{6pt}
We let $\DDD$ stand for the infinite-dimensional {\em differentiation matrix\/}
\begin{equation}
  \label{eq:2.2}
  \DDD_{m,n}=\int_a^b \varphi_m'(x)\varphi_n(x)\D x,\qquad m,n\in\BB{Z}_+.
\end{equation}
We say that $w$ is of {\em index\/} $s\in\BB{N}\cup\{\infty\}$ and denote this by $\CC{ind}\, w=s$ if $\DDD^{\,k}$ is bounded for $k=1,\ldots,s$, while $\DDD^{\,s+1}$ is unbounded.

\begin{lemma}
  $\DDD$ is skew-symmetric if and only if $w(a)=w(b)=0$.
\end{lemma}

\begin{proof}
  Assume first that $-\infty<a<b<\infty$ and note that $\DDD$ is skew-symmetric if and only if $\DDD_{m,n}+\DDD_{n,m}=0$, $m,n\in\BB{Z}_+$. Since
  \begin{displaymath}
     \varphi_n'=\sum_{k=0}^\infty \DDD_{n,k} \varphi_k,\qquad n\in\BB{Z}_+,
  \end{displaymath}
  it follows from \R{eq:2.2} and the orthonormality of $\Phi$ that $\DDD$ is skew symmetric if
  \begin{displaymath}
    \int_a^b \frac{\D \sqrt{w(x)} p_m(x)}{\D x} \sqrt{w(x)}p_n(x) \D x+\int_a^b \sqrt{w(x)}p_m(x)\frac{\D \sqrt{w(x)}p_n(x)}{\D x} \D x=0
  \end{displaymath}
  for $m,n\in\BB{Z}_+$.   The latter is equivalent, for every $m,n\in\BB{Z}_+$, to
  \begin{Eqnarray*}
    &&\int_a^b \left[ \frac{w'(x)}{2\sqrt{w(x)}} p_m(x)+\sqrt{w(x)}p_m'(x)\right]\! \sqrt{w(x)}p_n(x)\D x\\
    &&\mbox{}+\int_a^b \sqrt{w(x)}p_m(x)\!\left[\frac{w'(x)}{2\sqrt{w(x)}} p_n(x)+\sqrt{w(x)}p_n'(x)\right]\!\D x=0\\
    \Leftrightarrow&&\int_a^b [w'(x)p_m(x)p_n(x)+w(x)p_m'(x)p_n(x)+w(x)p_m(x)p_n'(x)]\D x=0\\
    \Leftrightarrow&&\int_a^b \left[w'(x)p_m(x)p_n(x)+w(x)\frac{\D p_m(x)p_n(x)}{\D x}\right]\D x\\
    \Leftrightarrow&&\int_a^b w'(x)p_m(x)p_n(x)\D x +w(x)p_m(x)p_n(x)\,\rule[-6pt]{0.75pt}{20pt}_{\,a}^{\,b}-\int_a^b w'(x)p_m(x)p_n(x)\D x\\
    \Leftrightarrow&&w(b)p_m(b)p_n(b)=w(a)p_m(a)p_n(a).
  \end{Eqnarray*}
  All the zeros of orthogonal polynomials reside in $(a,b)$. Therefore they cannot vanish at the endpoints and $(-1)^k p_k(a)p_k(b)>0$, $k\in\BB{N}$, hence $p_m(a)p_n(a)$ cannot equal $p_m(b)p_n(b)$ for all $m,n\in\BB{N}$. We deduce that $\DDD$ is skew-symmetric if and only if $w(a)=w(b)=0$. 
  
  The proof is similar -- in fact, somewhat simpler -- once either $b=\infty$ or $a=-\infty$. If $(a,b)=(-\infty,\infty)$ then $\CC{L}_2$ boundedness and continuity imply  $w(\pm\infty)=0$, $\DDD$ is skew symmetric and there is nothing to prove.
\end{proof}

Consequently, we impose an additional condition on the weight function, namely that it vanishes at the endpoints. Note that this is automatically true once an endpoint is infinite.

A quintessential example of a W-function are {\em Hermite functions\/}
\begin{displaymath}
  \varphi_n(x)=\frac{\ee^{-x^2/2}}{\sqrt{2^nn!\sqrt{\pi}}} \CC{H}_n(x),\qquad n\in\BB{Z}_+,
\end{displaymath}
where the $\CC{H}_n$s are standard Hermite polynomials. Hermite functions are well known in mathematical physics, because they are eigenfunctions of the free Schr\"odinger operator. They can be derived from orthonormalised Hermite polynomials via the Fourier transform route, as mentioned in the introduction and e.g.\ in \cite{iserles19oss}, hence their differentiation matrix is tridiagonal. On the other hand they are W-functions with $w(x)=\ee^{-x^2}$. Tridiagonality implies that $\CC{ind}\,w=\infty$. More generally, $\CC{ind}\,w=\infty$ once $\DDD$ is a banded matrix and it is interesting to characterise all weight functions with this feature.

\begin{theorem}
  The differentiation matrix $\DDD$ of a system of  W-functions is banded if and only if $w(x)=\ee^{-c(x)}$, $x\in\BB{R}$, where $c$ is an even-degree polynomial whose highest degree coefficient is strictly positive.
\end{theorem}

\begin{proof}
  Letting $m\leq n-1$, orthogonality implies that
  \begin{displaymath}
    \DDD_{m,n}=\int_a^b w\!\left(\frac12 \frac{w'}{w}p_m+p_m'\right)\!p_n\D x=\frac12 \int_a^b w'p_m p_n\D x
  \end{displaymath}
  while for $m\geq n+1$ skew symmetry yields
  \begin{equation}
    \label{eq:2.3}
    \DDD_{m,n}=-\frac12 \int_a^b w'p_m p_n\D x.
  \end{equation}
  Assume that $\DDD$ has bandwidth $2L+1$, in other words that $\DDD_{m,n}=0$ for $|m-n|\geq L+1$. It follows for $m\leq n-1$ that
  \begin{displaymath}
    \int_a^b w X_m  p_n\D x=0,\quad n\geq m+L+1\qquad\mbox{where}\qquad X_m=\frac{w'}{w}p_m. 
  \end{displaymath}
  If in addition $n\geq L$, expanding $X_m$ in the basis $\PPP$ it follows that $X_m$ is a polynomial of degree $m+L+1$. However,
  \begin{displaymath}
    w'=\frac{X_m}{p_m}w\qquad\rightarrow\qquad w(x)=w(x_0)\exp\left(\int_{x_0}^x \frac{X_m(y)}{p_m(y)}\D y\right)
  \end{displaymath}
  for some $x_0$. Since $w$ is independent of $m$, necessarily $p_m$ divides $X_m$ and the remainder $c$ is a polynomial independent of $m$. Therefore without loss of generality $w(x)=\ee^{-c(x)}$ where $c$ is a polynomial of degree $L+1$.  $w$ being integrable and $w(a)=w(b)=0$, necessarily $a=-\infty$, $b=\infty$,\footnote{In a finite interval we would have had an essential singularity at an endpoint.} $L$ is odd and $c$ is an even-degree polynomial with strictly positive  leading-degree coefficient.
\end{proof}

We have recovered {\em precisely\/} the W-functions associated with generalised Freud polynomials that have been originally considered in \cite{luong23apw}. However, such W-functions are of little interest within the context of this paper, since we seek weight functions of finite index. 

This is the point to note the expression \R{eq:2.3} for the elements of $\DDD$ such that $m\geq n+1$. (If $m\leq n-1$ we need to flip the sign.) We will make much use of it in the sequel. 

\subsection{The boundedness of $\DDD^{\,s}$}

We assume in this section that the weight $w$ is strictly positive in $(a,b)$, as smooth in $[a,b]$ as needed in our construction, and set
\begin{displaymath}
  q_j(x)=\frac{\D^j \sqrt{w(x)}}{\D x^j},\qquad j\in\BB{Z}_+.
\end{displaymath}
Therefore
\begin{displaymath}
  \varphi_m^{(\ell)}=\sum_{j=0}^\ell {\ell\choose j} q_j p_m^{(\ell-j)},\qquad \ell,m\in\BB{Z}_+.
\end{displaymath}
As long as $\DDD^{\,s}$ is bounded, we have
\begin{equation}
  \label{eq:2.4}
  (\DDD^{\,s})_{m,n}=\int_a^b \varphi_m^{(s)}\varphi_n\D x=\sum_{j=0}^s {s\choose j} \int_a^b \sqrt{w} q_j p_m^{(s-j)}\D x.
\end{equation}
It is trivial to prove that
\begin{displaymath}
  q_r=\sum_{j=1}^r \frac{U_{r,j}}{w^{j-\frac12}},
\end{displaymath}
where each $U_{r,j}$ is a linear combination of products of the form $\prod_i w^{(\ell_i)}$ such that $\ell_i\geq1$ and $\sum_i \ell_i=r$: for example 
\begin{displaymath}
  U_{4,1}=\frac12 w^{(4)},\qquad U_{4,2}=-\frac34 {w''}^2-w'w''',\qquad U_{4,3}=\frac94 {w'}^2w'',\qquad U_{4,4}=-\frac{15}{16}{w'}^4.
\end{displaymath}
In general, $U_{r,r}=(-1)^{r-1} (2r)! {w'}^r/(4^rr!)$, $r\in\BB{N}$.

Since $w(x)>0$ in $(a,b)$, the only possible source of singularity in \R{eq:2.4} is that $\varphi_m^{(s)}$ is non-integrable at an endpoint. Recalling that $w(a)=w(b)=0$, the source of any integrability is a division by a power of $w$ and, the larger the power, the more significant the singularity. In other words,  $\varphi_m^{(s)}$ is bounded for all $m\in\BB{Z}_+$ only if the integral
\begin{displaymath}
  \int_a^b \sqrt{w} q_s p\D x
\end{displaymath}
is bounded for any polynomial $p$, and this is contingent on $\tilde{w}_s={w'}^s/w^{s-1}$ being a {\em signed\/} weight function, i.e.\ all its moments exist and $\tilde{w}_s\not\equiv0$.\footnote{$\tilde{w}_1$ cannot be a `true' weight function because $w(a)=w(b)=0$, hence $w'$ changes sign in $(a,b)$.} 
The following theorem is thereby true.

\begin{theorem}
  A necessary condition for $\CC{ind}\,w\geq s$ is that $\tilde{w}_r$, $r=2,\ldots,s$, are signed weights.
\end{theorem}

Let $-\infty<a<b<\infty$. Regularity and $w(a)=w(b)=0$ imply that
\begin{equation}
  \label{eq:2.5}
  w(x)=(x-a)^\alpha (b-x)^\beta v(x),\qquad x\in[a,b],\quad v(a),v(b)\neq0.
\end{equation}
Therefore, after elementary algebra,
\begin{displaymath}
  \tilde{w}_s=(x-a)^{\alpha-s}(b-x)^{s-\beta} v \left[(\alpha b+\beta a)-(\alpha+\beta)x+(x-a)(b-x)\frac{v'}{v}\right]^{\!s}\!.
\end{displaymath}

\begin{theorem}
  A necessary condition for $\CC{ind}\, w\geq s$ in a finite interval $(a,b)$ is that $\alpha,\beta>s-1$. Likewise, once $b=\infty$, we need $\alpha>s-1$ and for $a=-\infty$ the condition is $\beta>s-1$.
\end{theorem}

\begin{proof}
  Consistently with our assumptions, $v\neq0$ in $[a,b]$, therefore the only source of singularity may come from $(x-a)^{\alpha-s}$ and $(b-x)^{\beta-s}$. We conclude that, for $\DDD^{\,s}$ to be bounded, we need $\alpha,\beta>s-1$. The semi-infinite cases follow in an identical (and simpler!) manner.
\end{proof}

In the special case $s=2$ we can complement Theorem~3 with a sufficient condition.

\begin{theorem}
  A necessary and sufficient condition for $\CC{ind}\,w\geq2$ is that $\tilde{w}_2={w'}^2/w$ is a signed measure.
\end{theorem}

\begin{proof}
  We compute $\DDD^{\,2}$ directly. Using skew symmetry,
  \begin{displaymath}
  \DDD_{m,n}^{\,2}=\sum_{\ell=0}^\infty \DDD_{m,\ell}\DDD_{\ell,n}=-\sum_{\ell=0}^{n-1} \DDD_{m,\ell}\DDD_{n,\ell} +\sum_{\ell=n+1}^{m-1} \DDD_{m,\ell}\DDD_{\ell,n} -\sum_{\ell=m+1}^\infty \DDD_{\ell,m}\DDD_{\ell,n}.
\end{displaymath}

Recalling \R{eq:2.3}, let us consider the infinite sum
\begin{Eqnarray*}
  -\sum_{\ell=m+1}^\infty \DDD_{\ell,m}\DDD_{\ell,n}&=&-\frac14 \int_a^b \int_a^b w'(x)w'(y)p_m(x)p_n(y)\sum_{\ell=m+1}^\infty p_\ell(x)p_\ell(y) \D x\D y\\
  &=&-\frac14 \int_a^b \int_a^b w'(x)w'(y)p_m(x)p_n(y)\sum_{\ell=0}^\infty p_\ell(x)p_\ell(y) \D x\D y\\
  &&\mbox{}+\frac14 \sum_{\ell=0}^m  \int_a^b w'(x)p_m(x) p_\ell(x)\D x\int_a^b w'(y)p_n(y) p_\ell(y) \D y\\
  &=&-\frac14 \int_a^b \int_a^b w'(x)w'(y)p_m(x)p_n(y)\sum_{\ell=0}^\infty p_\ell(x)p_\ell(y) \D x\D y\\
  &&\mbox{}+ \sum_{\ell=0}^{n-1} \DDD_{m,\ell}\DDD_{n,\ell} -\sum_{\ell=m+1}^{n-1} \DDD_{m,\ell}\DDD_{\ell,n}.
\end{Eqnarray*}
Therefore,
\begin{equation}
  \label{eq:2.6}
  \DDD_{m,n}^{\,2}=-\frac14 \int_a^b\! \int_a^b w'(x)w'(y)p_m(x)p_n(y)\sum_{\ell=0}^\infty p_\ell(x)p_\ell(y) \D x\D y.
\end{equation}

Let $\PPP$ be orthonormal and complete in $\CC{L}_2((a,b),w\D x)$ and  $f\in\CC{L}_2((a,b),w\D x)$. Then
\begin{displaymath}
  f(x)=\sum_{m=0}^\infty \hat{f}_m p_m(x),\qquad \mbox{where}\qquad \hat{f}_m=\int_a^b w(x)f(x)p_m(x)\D x.
\end{displaymath}
Moreover, by the Parseval theorem,
\begin{equation}
  \label{eq:2.7}
  \int_a^b w(x)|f(x)|^2\D x=\|f\|^2=\sum_{m=0}^\infty |\hat{f}_m|^2.
\end{equation}

Since 
\begin{displaymath}
  |\hat{f}_m|^2=\int_a^b \!\int_a^b w(x)w(y)f(x)\overline{f(y)} p_m(x)p_m(y)\D x\D y,
\end{displaymath}
exchanging summation and integration we have 
\begin{displaymath}
  \|f\|^2=\int_a^b\!\int_a^b w(x)w(y)f(x)\overline{f(y)} \sum_{m=0}^\infty p_m(x) p_m(y) \D x\D y.
\end{displaymath}
Let
\begin{displaymath}
  K(x,y)=\sqrt{w(x)w(y)} \sum_{m=0}^\infty p_m(x)p_m(y),
\end{displaymath}
the {\em Christoffel--Darboux kernel.\/} It now follows from \R{eq:2.7} that for every $f\in\CC{L}_2((a,b),w\D x)$ it is true that
\begin{displaymath}
  \int_a^b w(x) |f(x)|^2\D x=\int_a^b\int_a^b \sqrt{w(x)w(y)} f(x)\overline{f(y)} K(x,y)\D x\D y
\end{displaymath}
and we deduce that
\begin{equation}
  \label{eq:2.8}
  K(x,y)=\delta_{x-y}.
\end{equation}
In other words, $K$ is a reproducing kernel.\footnote{While this is probably known, the author failed to find this result in literature, even in the encyclopaedic review of the Christoffel--Darboux kernel in \cite{simon08cdk}, see also \cite{ismail05cqo,lasserre22cdk} -- the reason might well be that the emphasis is usually on general Borel measures, rather than on $\D\mu=w\D x$ with $w\in\CC{C}^1(a,b)$. One way or the other, the proof is included for completeness.}

We now return to \R{eq:2.6}, deducing that 
\begin{Eqnarray*}
  \DDD_{m,n}^{\,2}&=&-\frac14 \int_a^b \int_a^b \frac{w'(x)w'(y)}{\sqrt{w(x)w(y)}} p_m(x)p_n(y) K(x,y)\D x\D y\\
  &=&-\frac14 \int_a^b \frac{{w'}^2(x)}{w(x)} p_m(x)p_n(x) \D x=-\frac14 \int_a^b \tilde{w}_2(x) p_m(x)p_n(x) \D x.
\end{Eqnarray*}
This is bounded because $\tilde{w}_2$ is a signed measure and $p_mp_n$ a polynomial, and we deduce that $\CC{ind}\,w\geq2$. The necessity of $\tilde{w}_2$ being a signed measure is obvious from the argument that led to Theorem~4.
\end{proof}

\setcounter{equation}{0}
\setcounter{figure}{0}
\section{Separable systems}

In this section we consider two families of weight functions that share a hugely beneficial feature of {\em separability\/}, and we also provide two examples of weights that lack this feature.

We say that a weight function $w$ is {\em separable\/} if there exist real sequences $\GG{a}=\{\GG{a}_n\}_{n\in\bb{Z}_+}$ and $\GG{b}=\{\GG{b}_n\}_{n\in\bb{Z}_+}$ such that
\begin{equation}
  \label{eq:3.1}
  \DDD_{m,n}=
  \begin{case}
    \GG{a}_m\GG{b}_n, & m\geq n+1,\\[4pt]
    0, & m=n,\\[4pt]
    -\GG{a}_n\GG{b}_m, & m\leq n-1,
  \end{case}\qquad m,n\in\BB{Z}_+
\end{equation}
and it is {\em symmetrically separable\/} subject to the existence of real sequences $\GG{a}=\{\GG{a}_n\}_{n\in\bb{Z}_+}$ and $\GG{b}=\{\GG{b}_n\}_{n\in\bb{Z}_+}$ such that
\begin{equation}
  \label{eq:3.2}
  \DDD_{m,n}=
  \begin{case}
    -\GG{a}_m\GG{b}_n & m+n\mbox{\ odd,}\; m\geq n+1,\\[4pt]
    0, & m+n\mbox{\ even},\\[4pt]
    \GG{a}_m\GG{b}_n, & m+n\mbox{\ odd,}\; m\leq n-1,
  \end{case}\qquad m,n\in\BB{Z}_+
\end{equation}
It will be demonstrated in Section~4 that separability or symmetric separability allow for very rapid computation of products of the form $\DDD_N\MM{v}$ for $\MM{v}\in\BB{R}^{N+1}$. 

In this section we consider two families of measures, one separable and the other symmetrically separable: the {\em Laguerre weight\/} $w(x)=x^\alpha \ee^{-x}\chi_{(0,\infty)}(x)$ and the {\em ultraspherical weight\/} $(1-x^2)^\alpha\chi_{(-1,1)}(x)$ respectively. In a way, they are the most obvious measures in intervals of the form $(0,\infty)$ and $(-1,1)$ respectively. Yet, interestingly, separability appears to be a very rare feature and we provide counterexamples further in this section. 

In both Laguerre and ultraspherical cases we are able to present comprehensive analysis, deriving the sequences $\GG{a},\GG{b}$ explicitly, determining $\CC{ind}\,w$ and (in Section~4) discussing the optimal choice of the parameter $\alpha$. 

\subsection{The Laguerre family}

Laguerre polynomials are orthogonal with respect to the Laguerre weight,
\begin{displaymath}
  \int_0^\infty x^\alpha \ee^{-x} \LL_m^{(\alpha)}(x)\LL_n^{(\alpha)}(x)\D x=\frac{\G(n+1+\alpha)}{n!}\delta_{m,n},\qquad m,n\in\BB{Z}_+,\quad \alpha>-1.
\end{displaymath}
\cite[p.~206]{rainville60sf}. In our case we consider just the case $\alpha>0$, so that the weight function vanishes at the origin. We have
\begin{Eqnarray*}
  p_n(x)&=&\sqrt{\frac{n!}{\G(n+1+\alpha)}} \LL_n^{(\alpha)}(x),\\
  \varphi_n(x)&=&\sqrt{\frac{n!}{\G(n+1+\alpha)}} x^{\alpha/2}\ee^{-x/2}\LL_n^{(\alpha)}(x),\qquad n\in\BB{Z}_+.
\end{Eqnarray*}

In Theorem \ref{th:Laguerre} we determine that the Laguerre weight is separable -- the proof requires a fair bit of algebraic computation and is relegated to Appendix~A. The separability coefficients are given in \R{eq:A.3}, which we repeat here for clarity,
\begin{equation}
    \label{eq:3.3}
    \GG{a}_m=\sqrt{\frac{m!}{2\G(m+1+\alpha)}}\sim \frac{1}{m^{\alpha/2}},\quad \GG{b}_n=\sqrt{\frac{\G(n+1+\alpha)}{2n!}}\sim n^{\alpha/2},\qquad m,n\in\BB{Z}_+.
\end{equation}
Note that
\begin{equation}
  \label{eq:3.4}
  \GG{a}_m\GG{b}_m\equiv \frac12,\qquad m\in\BB{Z}_+
\end{equation}
-- this will be important in the sequel. 

Theorem~4 presents a necessary condition for $\CC{ind}\, w\geq s$ for $s\geq2$: for a Laguerre weight $w=w_\alpha$ it translates to $\alpha>s-1$. In the remainder of this subsection we wish to prove that for the Laguerre weight function this condition is also sufficient. 

The matrix $\DDD^{\,s}$ is {\em absolutely bounded\/} for $s\geq0$ if 
\begin{equation}
    \label{eq:3.5}
    \sum_{k_1=0}^\infty \sum_{k_2=0}^\infty \cdots \sum_{k_{s-1}=0}^\infty |\DDD_{m,k_1}\DDD_{k_1,k_2}\cdots \DDD_{k_{s-2},k_{s-1}} \DDD_{k_{s-1},n}|<\infty, \qquad m,n\in\BB{Z}_+.
\end{equation}
It is clear that absolute boundedness implies boundedness. 

We assume that $m\geq n+1$ and observe that everything depends on the interplay of the relative sizes of  $k_0=m,k_1,k_2,\ldots,k_{s-1},k_s=n$ because, for example,
\begin{displaymath}
  k_j>k_{j+1}\quad\Rightarrow\quad |\DDD_{k_j,k_{j+1}}|=\GG{a}_{k_j}\GG{b}_{k_{j+1}},\qquad k_j<k_{j+1}\quad\Rightarrow\quad |\DDD_{k_j,k_{j+1}}|=\GG{a}_{k_{j+1}}\GG{b}_{k_j}.
\end{displaymath}
We can disregard the case $k_j=k_{j+1}$ because then $\DDD_{k_j,k_{j+1}}=0$ and the entire product vanishes, hence we assume that always $k_j\neq k_{j+1}$. We use the shorthand $\rA$ for $k_j>k_{j+1}$ and $\lA$ for $k_j<k_{j+1}$. Note that, once $s$ is even then $\DDD^{\,s}$ is symmetric and diagonal elements no longer vanish: in that case we need to consider also the case $m=n$ but the proof is identical.

To illustrate our argument, for $s=4$ we have eight options:\\[3pt]
\begin{tabular}{llcll}
    $\rA\rA\rA=\rA^3$:& $k_0>k_1>k_2>k_3,$ &\qquad\qquad &$\rA\lA\rA$:& $k_0>k_1<k_2>k_3,$\\
    $\rA\rA\lA=\rA^2\lA$:& $k_0>k_1>k_2<k_3,$ && $\rA\lA\lA=\rA\lA^2$:& $k_0>k_1<k_2<k_3,$\\
    $\lA\rA\rA=\lA\rA^2$:& $k_0<k_1>k_2>k_3,$ && $\lA\lA\rA=\lA^2\rA$:& $k_0<k_1<k_2>k_3,$\\
    $\lA\rA\lA$: & $k_0<k_1>k_2<k_3,$ && $\lA\lA\lA=\lA^3$: & $k_0<k_1<k_2<k_3,$
\end{tabular}\\[3pt]
except that $\lA^3$ is impossible because $k_0=m> n=k_2$.

We let $\mathcal{Q}_N$ stand for a generic polynomial of degree exactly $N$ and note for further use the following technical result with a straightforward proof.

\begin{proposition}
  The sum
  \begin{displaymath}
    \sum_{k=1}^K \frac{\mathcal{Q}_N(k)}{k^\alpha}\sim c K^{N-\alpha+1},\qquad K\gg1,\quad N\neq \alpha-1,
  \end{displaymath}
  converges as $K\rightarrow\infty$ if and only if $\alpha>N+1$. Here $c$ is a constant.
\end{proposition}

In general, the main idea is to write a sequence of $\lA$s and $\rA$s in the form 
\begin{displaymath}
  \lA^{i_1}\rA^{j_1}\lA^{i_2}\rA^{j_2}\cdots \lA^{i_t}\rA^{j_t},
\end{displaymath}
where $i_k,j_k\geq0$ and $\sum_{k=1}^t (i_k+j_k)=s$. We call $\lA^r$ a {$\lA$-pre-chain of length $r$\/} and $\rA^r$ a  {$\rA$-pre-chain of length $r$,\/} in other words we decompose each product in \R{eq:3.5} into a sequence of pre-chains. 

Consider first an $\lA$-pre-chain of length $r\geq1$. Because of \R{eq:3.4}, it equals 
\begin{Eqnarray*}
    &&\sum_{k_\ell=0}^{k_{\ell-1}-1} \sum_{k_{\ell+1}=0}^{k_\ell-1} \cdots \sum_{k_{\ell+r-1}=0}^{k_{\ell+r-2}-1} |\DDD_{k_{\ell-1},k_\ell} \DDD_{k_\ell,k_{\ell+1}}\cdots \DDD_{k_{\ell+r-2},k_{\ell+r-1}}|\\
    &=&\sum_{k_\ell=0}^{k_{\ell-1}-1} \sum_{k_{\ell+1}=0}^{k_\ell-1} \cdots \sum_{k_{\ell+r-1}=0}^{k_{\ell+r-2}-1} \prod_{j=\ell-1}^{\ell+r-2} \GG{a}_{k_j}\GG{b}_{k_{j+1}}=\frac{\GG{a}_{k_{\ell-1}}}{2^{r-1}} \sum_{k_\ell=0}^{k_{\ell-1}-1} \sum_{k_{\ell+1}=0}^{k_\ell-1} \cdots \sum_{k_{\ell+r-1}=0}^{k_{\ell+r-2}-1} \GG{b}_{k_{\ell+r-1}}.
\end{Eqnarray*}
We say that $\GG{a}_{k_{\ell-1}}$ and $\GG{b}_{k_{\ell+r-1}}$ are the {\em head\/} and the {\em tail\/} of the pre-chain, respectively.

Likewise, for an $\rA$-pre-chain of length $r\geq1$ we have 
 \begin{Eqnarray*}
    &&\sum_{\scriptscriptstyle k_\ell=k_{\ell-1}+1}^\infty \sum_{\scriptscriptstyle k_{\ell+1}=k_\ell+1}^\infty \cdots \hspace*{-6pt}\sum_{\scriptscriptstyle k_{\ell+r-1}=k_{\ell+r-2}+1}^\infty \hspace*{-16pt}  |\DDD_{k_{\ell-1},k_\ell} \DDD_{k_\ell,k_{\ell+1}}\cdots \DDD_{k_{\ell+r-2},k_{\ell+r-1}}|\\
    &=&\frac{\GG{b}_{k_{\ell-1}}}{2^{r-1}} \sum_{\scriptscriptstyle k_\ell=k_{\ell-1}+1}^\infty \sum_{\scriptscriptstyle k_{\ell+1}=k_\ell+1}^\infty \cdots \hspace*{-6pt}\sum_{\scriptscriptstyle k_{\ell+r-1}=k_{\ell+r-2}+1}^\infty \GG{a}_{k_{\ell+r-1}}.
\end{Eqnarray*}
Now $\GG{b}_{k_{\ell-1}}$ and $\GG{a}_{k_{\ell+r-1}}$ are the head and the tail of the pre-chain, respectively. 

Except for $\ell=0$ and $\ell+r=s$, we join the tail of a pre-chain to the head of the succeeding pre-chain. The outcome are $\lA$-chains and $\rA$-chains. Note thus that a chain has no head, while its tail is multiplied by the head of its successor pre-chain. (In this procedure we lose the head of the leading pre-chain and the tail of the last pre-chain but this makes no difference to the finiteness -- or otherwise -- of the sum)

An $\lA$-chain of length $r$ is of the form
\begin{displaymath}
  \frac{1}{2^{r-1}}\sum_{k_\ell=0}^{k_{\ell-1}-1} \sum_{k_{\ell+1}=0}^{k_\ell-1} \cdots \sum_{k_{\ell+r-1}=0}^{k_{\ell+r-2}-1} \GG{b}_{k_{\ell+r-1}}^2=\frac{1}{2^{r}}\sum_{k_\ell=0}^{k_{\ell-1}-1} \sum_{k_{\ell+1}=0}^{k_\ell-1} \cdots \sum_{k_{\ell+r-1}=0}^{k_{\ell+r-2}-1} \frac{\G(k_\ell+r+\alpha)}{(k_\ell+r-1)!},
\end{displaymath}
a finite sum. Hence, it cannot be a source for unboundedness of the sum \R{eq:3.5}. Matters are different, though, with an $\rA$-chain of length $r$: straightforward algebra and Proposition~6 imply that
\begin{Eqnarray*}
    &&\sum_{k_\ell=k_{\ell-1}+1}^\infty \sum_{k_{\ell+1}=k_\ell+1}^\infty \cdots \sum_{k_{\ell+r-1}=k_{\ell-r-2}+1}^\infty \GG{a}^2_{k_{\ell+r-1}}\\
    &=&\sum_{k_{\ell+r-1}=k_{\ell-1}+r}^\infty \GG{a}^2_{k_{\ell+r-1}} \sum_{k_\ell=k_{\ell-1}+r-1}^{k_{\ell+r-1}-r+1} \sum_{k_{\ell+1}=k_{\ell-1}+r-2}^{k_{\ell+r-1}-r+2}\cdots \sum_{k_{\ell+r-2}=k_{\ell-1}+1}^{k_{\ell+r-1}-1} \!\!1\\
    &=&\sum_{k_{\ell+r-1}=k_{\ell-1}+r}^\infty \GG{a}_{k_{\ell+r-1}}^2 \mathcal{Q}_{r-1}(k_{\ell+r-1}) \sim \sum_{\ell=k_{\ell-1}+r}^\infty \frac{1}{\ell^{\alpha-r}}.
\end{Eqnarray*}
Therefore boundedness takes place if $\alpha-r>1$.

Since the length of any chain is at most $s-1$ and $\rA^{s-1}$ is impossible (recall, $k_0>k_s$),  the maximal length of an $\rA$-chain is $s-2$. We thus deduce that $\alpha>s-1$. 

\begin{theorem}
   $\CC{ind}\,w_\alpha\geq s$ for the Laguerre weight if and only if $\alpha>s-1$.
\end{theorem}

\begin{proof}
  The necessity has been already proved in Theorem~4, while sufficiency follows because absolute boundedness in \R{eq:3.5} implies boundedness.
\end{proof}

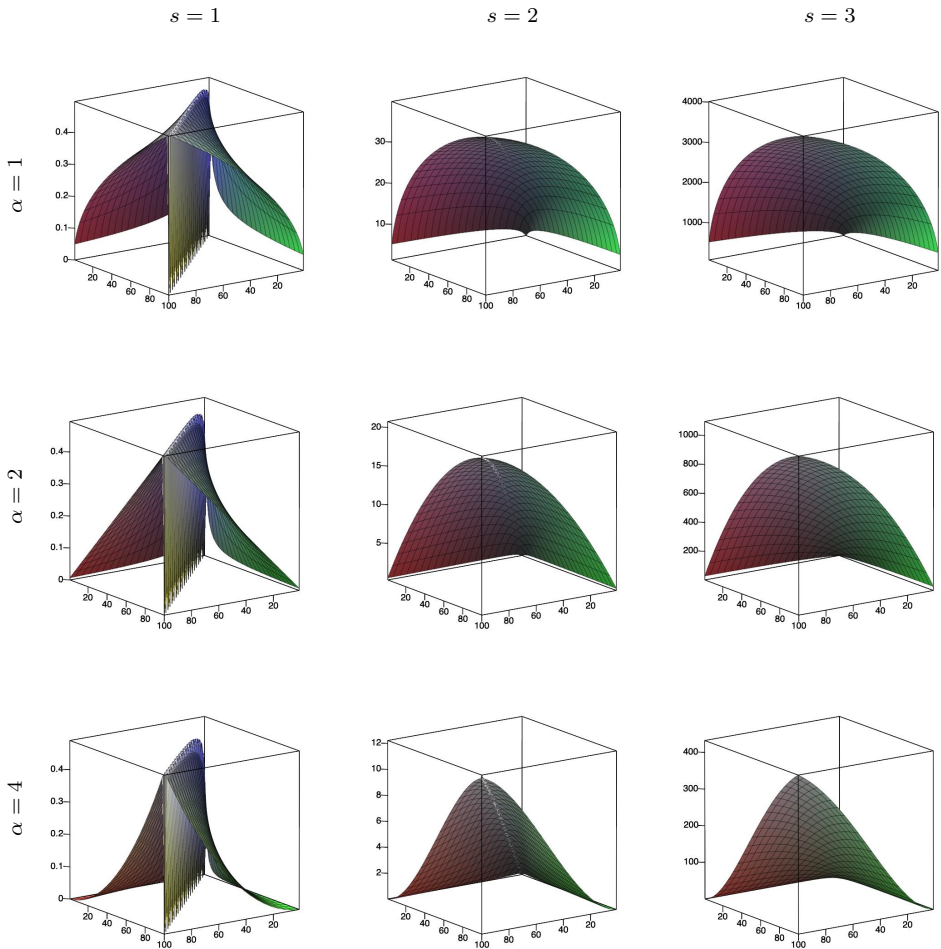
\begin{figure}[htb]
  \begin{center}
    \vspace*{25pt}
    \begin{picture}(0,0)
    \put (56,122) {\footnotesize$s=1$}
    \put (176,122) {\footnotesize$s=2$}
    \put (296,122) {\footnotesize$s=3$}
    \put (-5,50) {\rotatebox{90}{\footnotesize$\alpha=1$}}
    \put (-5,-68) {\rotatebox{90}{\footnotesize$\alpha=2$}}
    \put (-5,-186) {\rotatebox{90}{\footnotesize$\alpha=4$}}
    \put (-30,-130) {\rotatebox{90}{\footnotesize\bf Laguerre differentiation matrix}}
  \end{picture}
  \includegraphics[width=120pt]{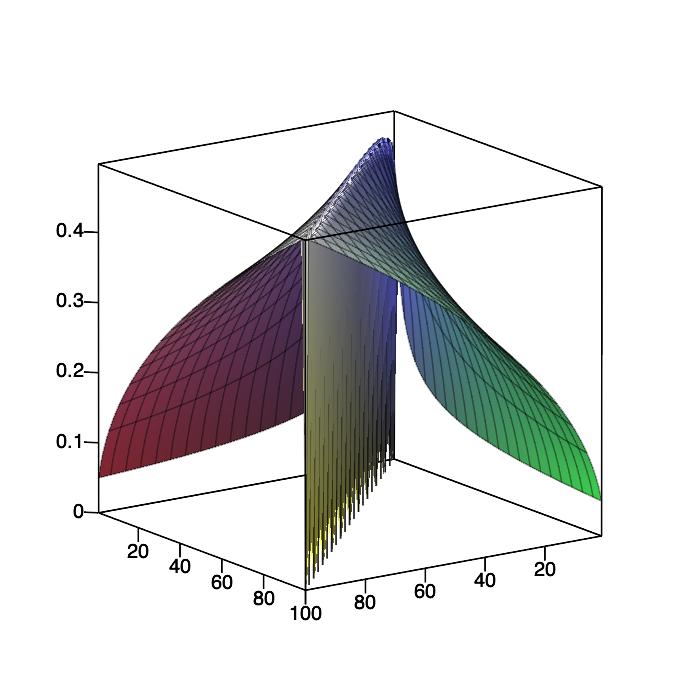}\hspace{0pt}\includegraphics[width=120pt]{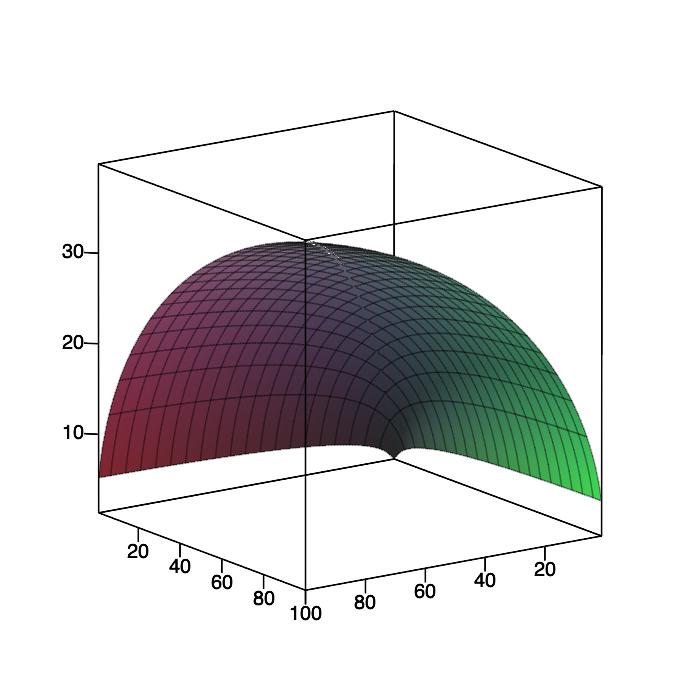}\hspace{0pt}\includegraphics[width=120pt]{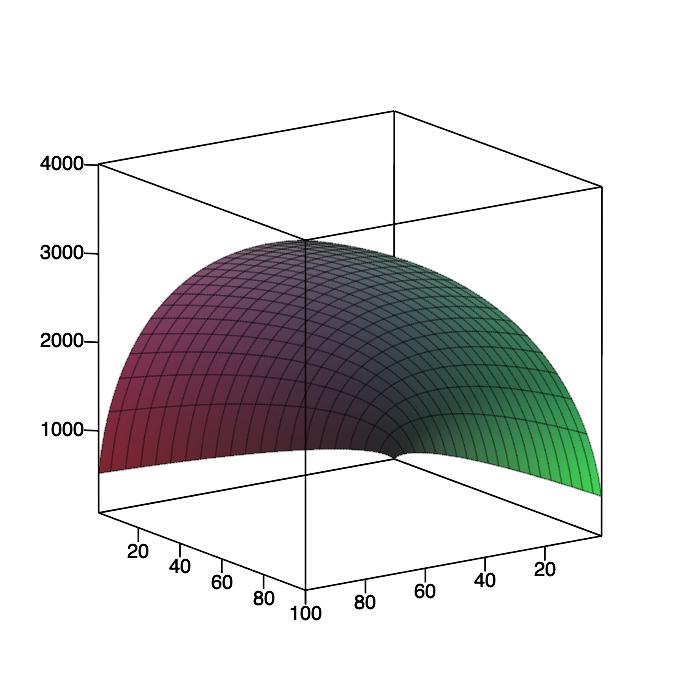} 
  \includegraphics[width=120pt]{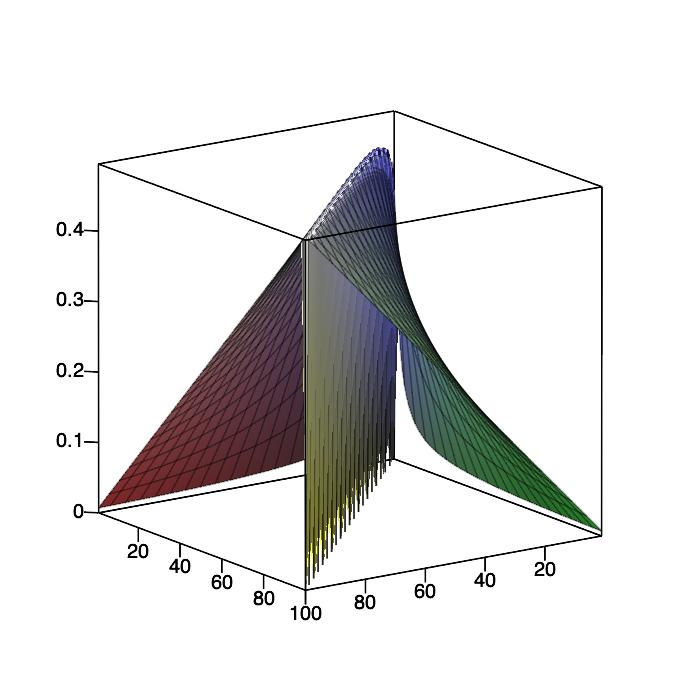}\hspace{0pt}\includegraphics[width=120pt]{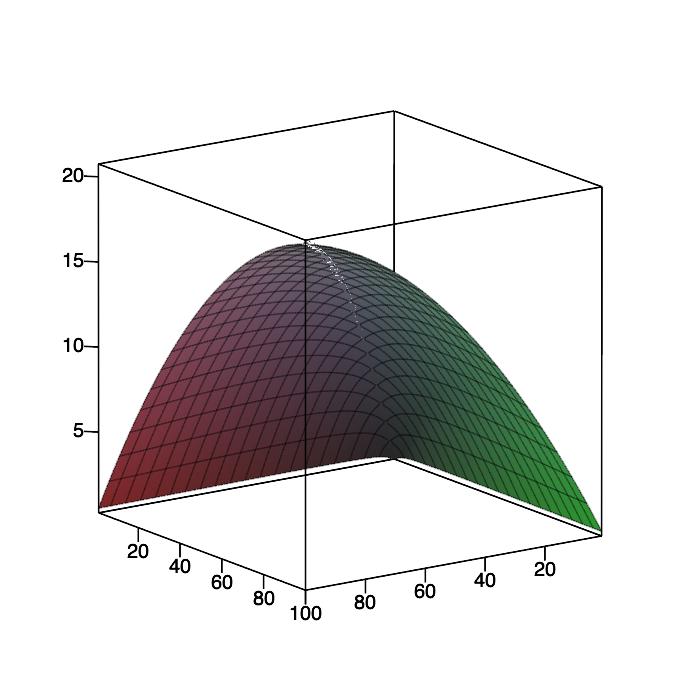}\hspace{0pt}\includegraphics[width=120pt]{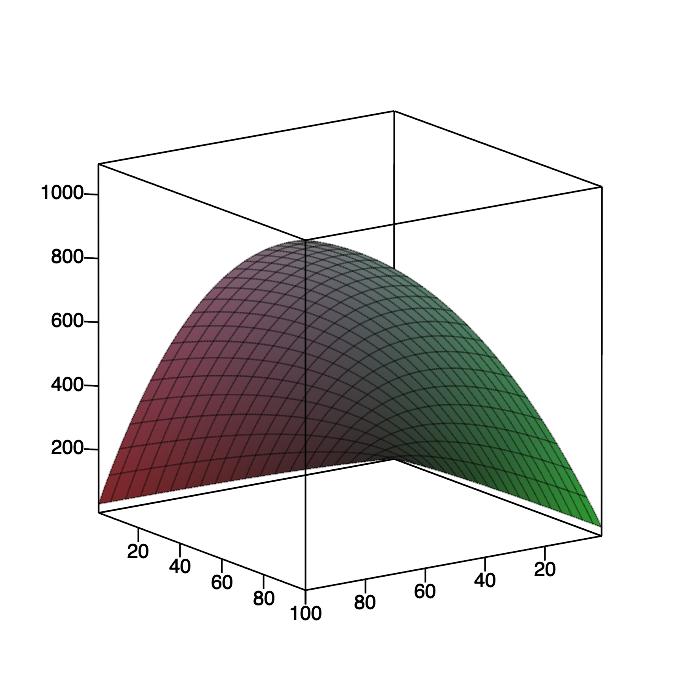} 
  \includegraphics[width=120pt]{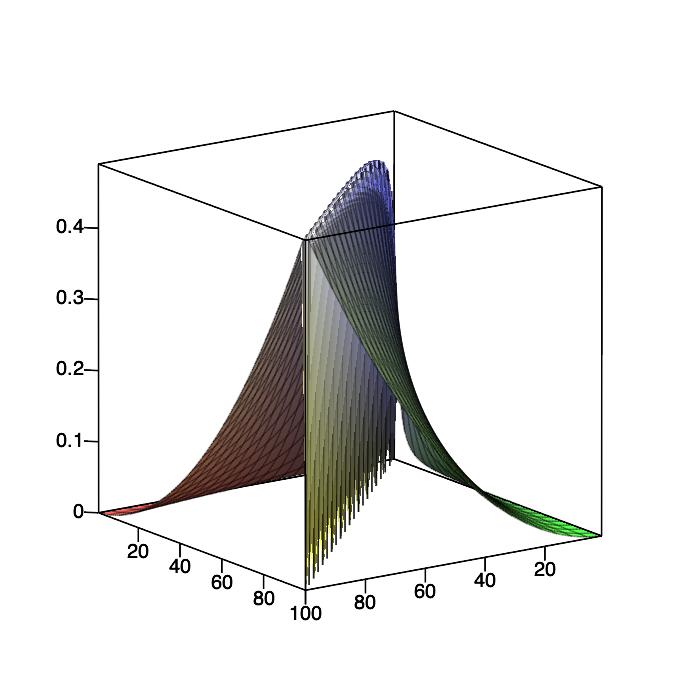}\hspace{0pt}\includegraphics[width=120pt]{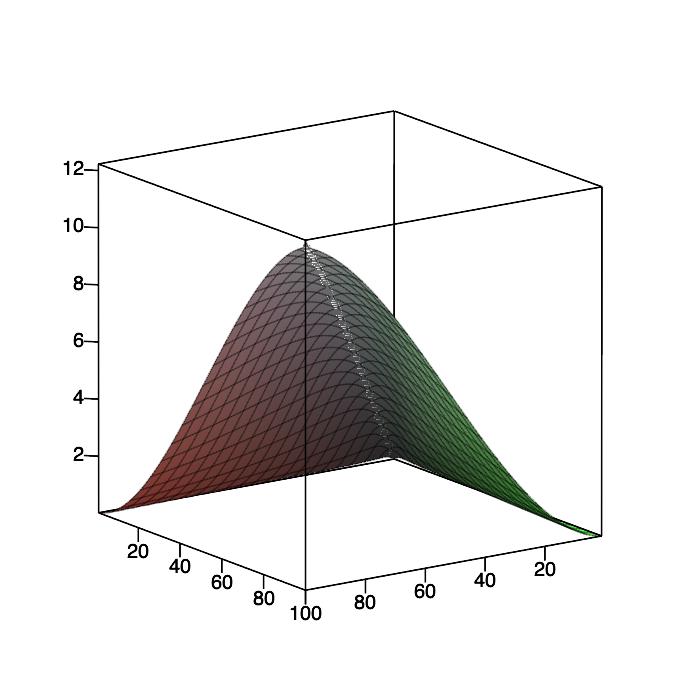}\hspace{0pt}\includegraphics[width=120pt]{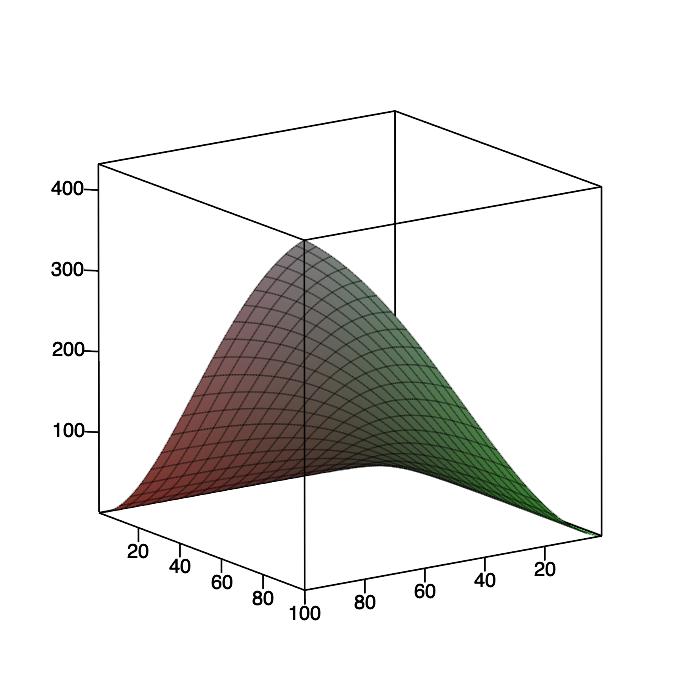}
    \caption{Laguerre W-functions: The magnitude of $\DDD^{\,s}$ for different values of $\alpha$ and $1\leq s\leq3$}
    \label{fig:3.1}
  \end{center}
\end{figure}

In Fig.~\ref{fig:3.1} we display the absolute values of the entries of $\DDD^{\,s}$ for different values of $\alpha$ and $s$. The computation involves infinite matrices, hence infinite products which need be truncated in computation. Thus, we compute $300\times300$ matrices and their powers, while displaying just their $100\times100$ section, since this minimises the truncation effects, For $s=1$ all differentiation matrices are bounded and of a moderate size, the sole difference is that, as $\alpha$ grows, the matrix becomes more `centred' about the diagonal. However, already for $s=2$ the difference is discernible. For $\alpha=1$ we are right on the boundary of $\alpha>s-1$ (on its wrong side!) and the size of $\DDD^{\,2}$ grows rapidly: had we displayed a section of an $M\times M$ matrix for $M\gg300$, the magnitude would have grown at a logarithmic rate, as indicated by the proof of absolute boundedness. Once $\alpha>1$, the magnitude grows at a slower rate and would remain bounded for $M\rightarrow\infty$. Finally, for $s=3$ the cases $\alpha=1$ and $\alpha=2$ correspond to polynomial and logarithmic growth, respectively, and this is apparent in the figure. Finally, for $\alpha=4$ the rate of growth slows down and it is persuasive that the magnitude remains bounded as $M\rightarrow\infty$. Note that even in a `good' $\alpha$ regime the magnitude, while decaying along rows and columns, grows along diagonals. We refer to the discussion following Fig.~\ref{fig:3.2} for an explanation of this behaviour, commenting here that this phenomenon follows from $\GG{a}_m\sim m^{-\alpha/2}$ and $\GG{b}_n\sim n^{\alpha/2}$.

It follows from Theorem~7 that, once we approximate functions in $\HHH{s}(0,\infty)$, we need to choose $\alpha>s-1$. However, there is much more to the choice of a good $\alpha$ and we defer its discussion to Section~4. As it turns out, the quality of approximation is exceedingly sensitive to the right choice and numerical results indicate that there exists a `sweet spot' that brings about substantially improved quality of approximation. 

\subsection{The ultraspherical family}

The ultraspherical weight\footnote{Also known, subject to different scaling, as the Gegenbauer weight \cite[p.~276]{rainville60sf}.}  a special case of the {\em Jacobi weight,\/} is $w_\alpha(x)=(1-x^2)^\alpha \chi_{(-1,1)}(x)$, $\alpha>1$ -- in our case  the requirement $w_\alpha(\pm1)=0$ restricts $\alpha$ to the range $(0,\infty)$. We have
\begin{Eqnarray*}
  p_n(x)&=&g_n^\alpha\PP_n^{(\alpha,\alpha)}(x),\\
  \varphi_n(x)&=&g_n^\alpha (1-x^2)^{\alpha/2} \PP_n^{(\alpha,\alpha)}(x),\qquad n\in\BB{Z}_+,
\end{Eqnarray*}
where the constant 
\begin{displaymath}
  g_n^\alpha=\frac{\sqrt{\frac12 n!(2n+2\alpha+1)\G(n+2\alpha+1)}}{2^\alpha\G(n+\alpha+1)}
\end{displaymath}
orthonormalises an ultraspherical polynomial \cite[p.~260]{rainville60sf}. Recalling the identity \R{eq:2.3}, we let
\begin{displaymath}
  \EEE_{m,n}=\frac{1}{\alpha g_m^\alpha g_n^\alpha}\DDD_{m,n}=\int_{-1}^1 (1-x^2)^{\alpha-1} x\PP_m^{(\alpha,\alpha)}(x)\PP_n^{(\alpha,\alpha)}(x)\D x.
\end{displaymath}
It is sufficient to derive the $\EEE_{m,n}$s explicitly and prove that $\EEE$ is symmetrically separable. We recall that our interest is in odd values of $m+n$ and  assume without loss of generality that $m\geq n+1$

Let
\begin{displaymath}
  S_{m,n}^\alpha=\int_{-1}^1 (1-x^2)^{\alpha-1}\PP_m^{(\alpha,\alpha)}(x) \PP_n^{(\alpha,\alpha)}(x)\D x,
\end{displaymath}
noting that $S_{m,n}^\alpha=0$ if $m+n$ is odd. The three-term recurrence relation for orthonormal ultraspherical polynomials is
\begin{equation}
  \label{eq:3.6}
  x\PP_m^{(\alpha,\alpha)}(x)=\frac{(m+1)(m+2\alpha+1)}{(m+1+\alpha)(2m+2\alpha+1)} \PP_{m+1}^{(\alpha,\alpha)}(x) +\frac{m+\alpha}{2m+2\alpha+1} \PP_{m-1}^{(\alpha,\alpha)}(x),
\end{equation}
as can be easily confirmed from \cite[p.~263]{rainville60sf}. Therefore 
\begin{Eqnarray}
  \nonumber
  &&\EEE_{m,n}\\
  \nonumber
  &=&\int_{-1}^1 (1-x^2)^{\alpha-1} \left[\frac{(m+1)(m+2\alpha+1)}{(m+1+\alpha)(2m+2\alpha+1)} \PP_{m+1}^{(\alpha,\alpha)} +\frac{m+\alpha}{2m+2\alpha+1} \PP_{m-1}^{(\alpha,\alpha)}\right]\! \PP_n^{(\alpha,\alpha)}\D x\\
  \label{eq:3.7}
  &=&\frac{(m+1)(m+2\alpha+1)}{(m+1+\alpha)(2m+2\alpha+1)} S_{m+1,n}^\alpha +\frac{m+\alpha}{2m+2\alpha+1} S_{m-1,n}^\alpha.
\end{Eqnarray}

Our next task is determining the explicit form of $S_{m,n}^\alpha$ for even $m+n$  and, without loss of generality, $m\geq n$. This is accomplished in Appendix~B and results in
\begin{displaymath}
  S_{m,n}^\alpha=\frac{4^\alpha}{\alpha} \frac{\G(m+1+\alpha)\G(n+1+\alpha)}{n!\G(m+1+2\alpha)},\qquad m\geq n,\; m+n\mbox{\ even}.
\end{displaymath}
We conclude from \R{eq:3.7} that 
\begin{displaymath}
  \EEE_{m,n}=\frac{4^\alpha}{\alpha} \frac{\G(m+1+\alpha)\G(n+1+\alpha)}{n!\G(m+1+2\alpha)},\qquad m\geq n,\; m+n\mbox{\ odd}
\end{displaymath}
and
\begin{equation}
  \label{eq:3.8}
  \DDD_{m,n}=\alpha g_m^\alpha g_n^\alpha \EEE_{m,n}=\frac12 \sqrt{\frac{m!(2m+2\alpha+1)(2n+2\alpha+1)\G(n+1+2\alpha)}{n!\G(m+1+2\alpha)}}
\end{equation}
is valid for all odd $m+n$, $m\geq n+1$ -- once $n\geq m+1$, we need to invert the sign. (Of course, $\DDD_{m,n}=0$ once $m+n$ is even.)

Our first conclusion is that the measure is symmetrically separable with 
\begin{equation}
  \label{eq:3.9}
  \GG{a}_m=\sqrt{\frac{m!(2m+2\alpha+1)}{2\G(m+1+2\alpha)}},\qquad \GG{b}_n=\sqrt{\frac{(2n+2\alpha+1)\G(n+1+2\alpha)}{2n!}}.
\end{equation}

The next conclusion is that the rate of growth (or decay) is dramatically different along the rows and the columns of $\DDD$. It follows from \R{eq:3.8} and the standard Stirling formula \cite[5.11.3]{dlmf} that
\begin{displaymath}
  \DDD_{m,n}\sim \frac{n^{\alpha+\frac12}}{m^{\alpha-\frac12}},\qquad m,n\gg1,\quad m\geq n+1.
\end{displaymath}
Therefore the elements of the differentiation matrix decay geometrically (at any rate, for $\alpha>\frac12$) along rows (and, because of skew symmetry, columns) and increase geometrically along diagonals. Note that, forming powers of $\DDD$, it is the decay along rows and columns that allows for  boundedness. (Incidentally, it can be proved using special functions  that $(\DDD^{\,2})_{0,0}=\alpha(2\alpha+1)/[4(\alpha-1)]$, driving home the fact, already known from Theorem~5, that $\alpha>1$ is necessary and sufficient for boundedness. We leave the proof, which plays no further role in our narrative, as an exercise for the reader.)

\begin{figure}[htb]
  \begin{center}
    \vspace*{25pt}
    \begin{picture}(0,0)
    \put (56,122) {\footnotesize$s=1$}
    \put (176,122) {\footnotesize$s=2$}
    \put (296,122) {\footnotesize$s=3$}
    \put (-5,50) {\rotatebox{90}{\footnotesize$\alpha=1$}}
    \put (-5,-68) {\rotatebox{90}{\footnotesize$\alpha=2$}}
    \put (-5,-186) {\rotatebox{90}{\footnotesize$\alpha=4$}}
    \put (-30,-130) {\rotatebox{90}{\footnotesize\bf Ultraspherical differentiation matrix}}
  \end{picture}
  \includegraphics[width=120pt]{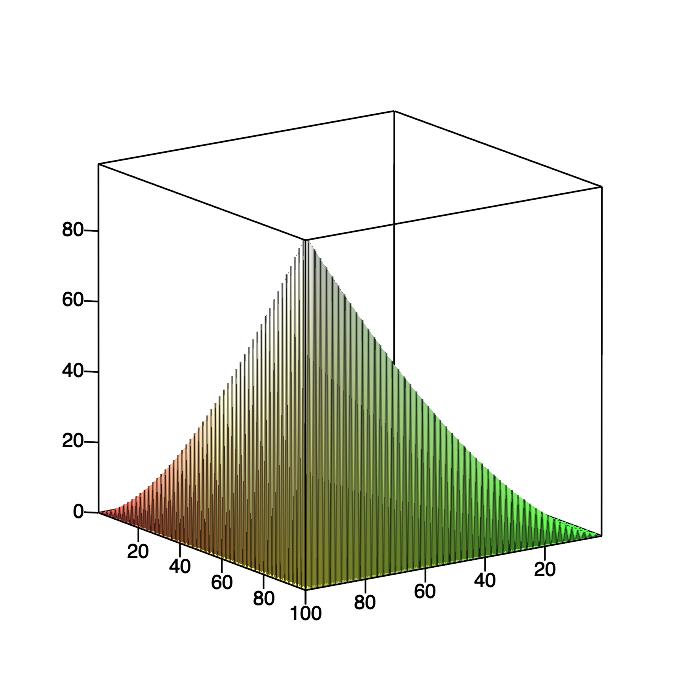}\hspace{0pt}\includegraphics[width=120pt]{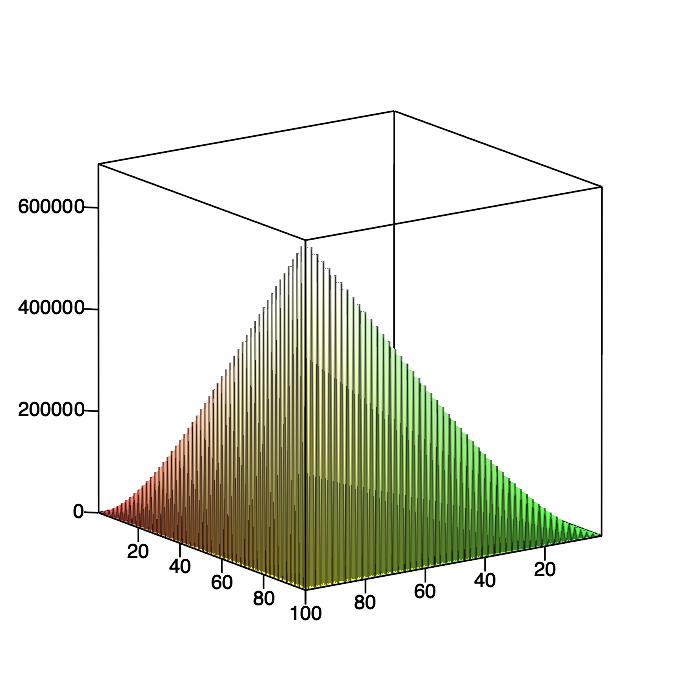}\hspace{0pt}\includegraphics[width=120pt]{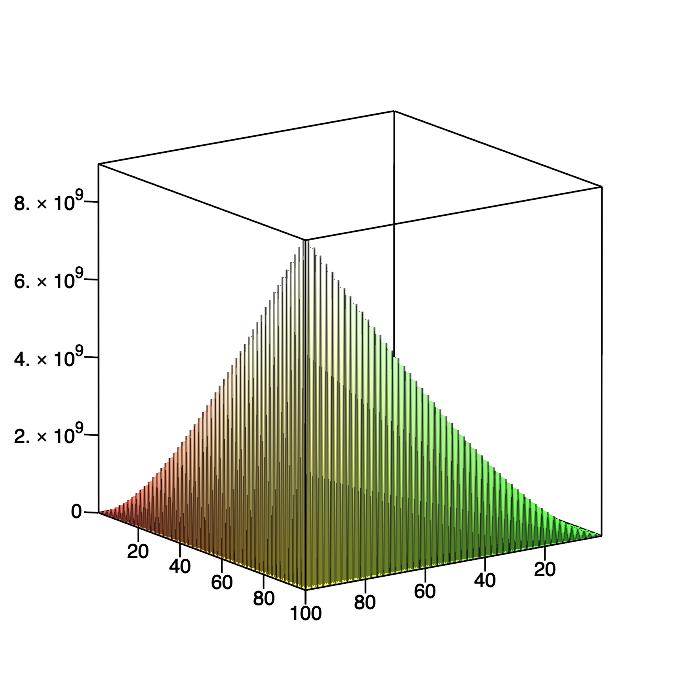} 
  \includegraphics[width=120pt]{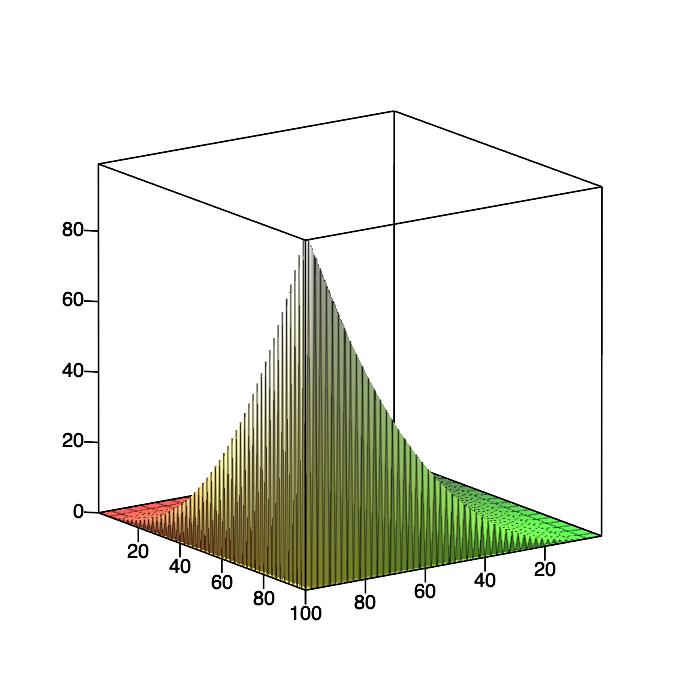}\hspace{0pt}\includegraphics[width=120pt]{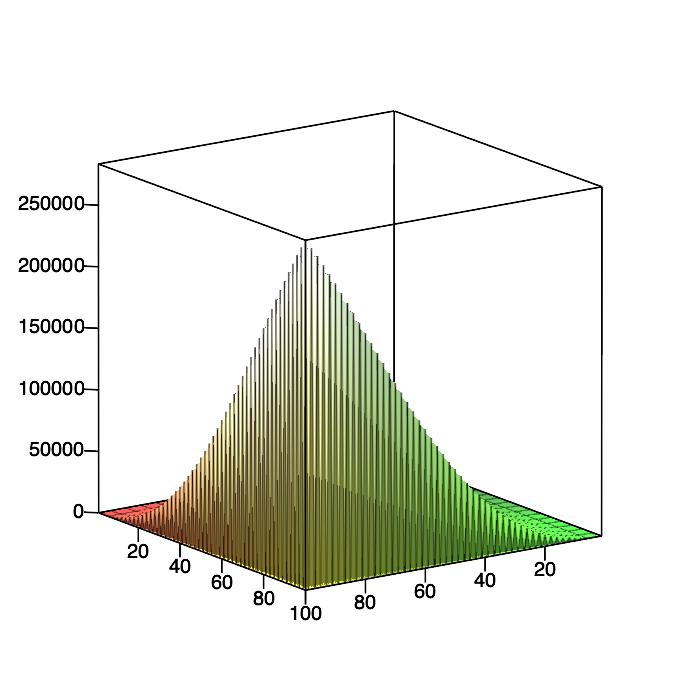}\hspace{0pt}\includegraphics[width=120pt]{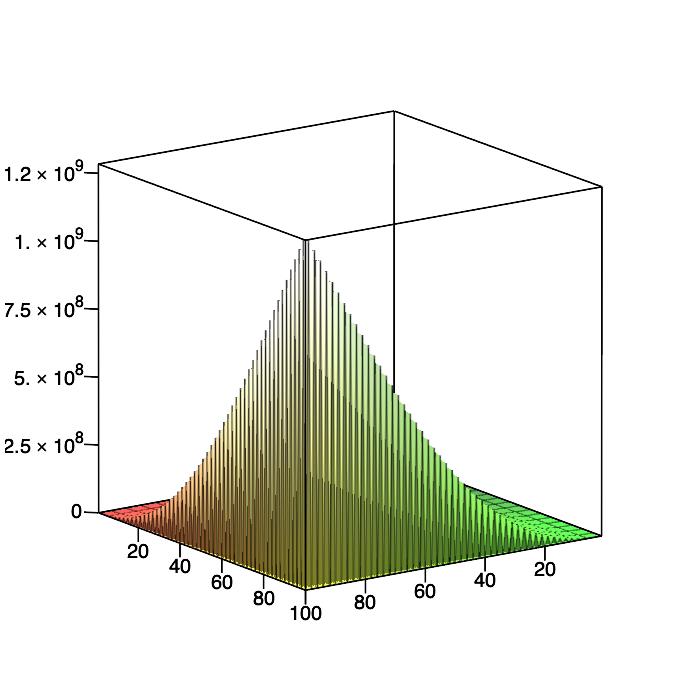} 
  \includegraphics[width=120pt]{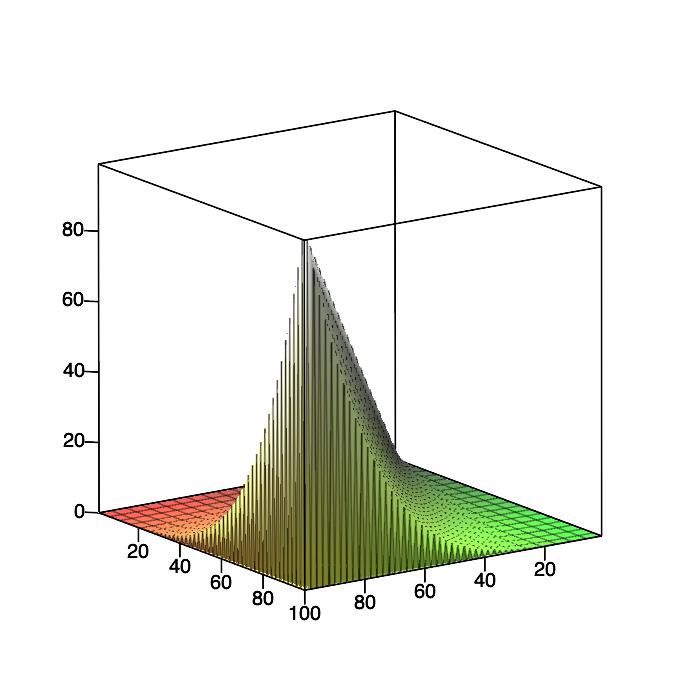}\hspace{0pt}\includegraphics[width=120pt]{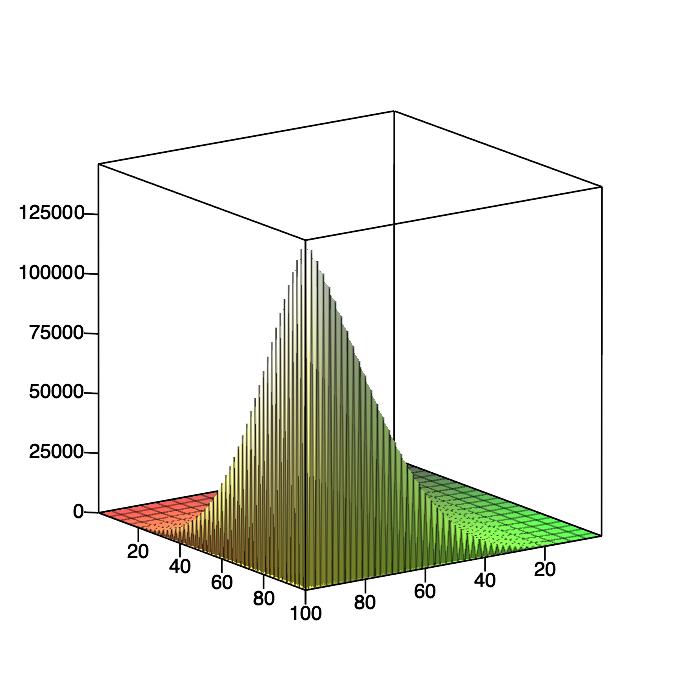}\hspace{0pt}\includegraphics[width=120pt]{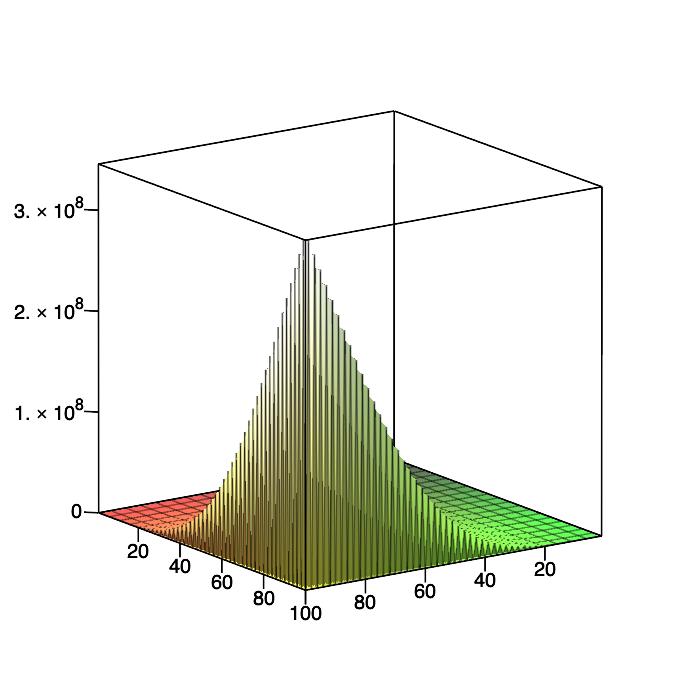}
    \caption{Ultraspherical W-functions: The magnitude of $\DDD^{\,s}$ for different values of $\alpha$ and $1\leq s\leq3$}
    \label{fig:3.2}
  \end{center}
\end{figure}
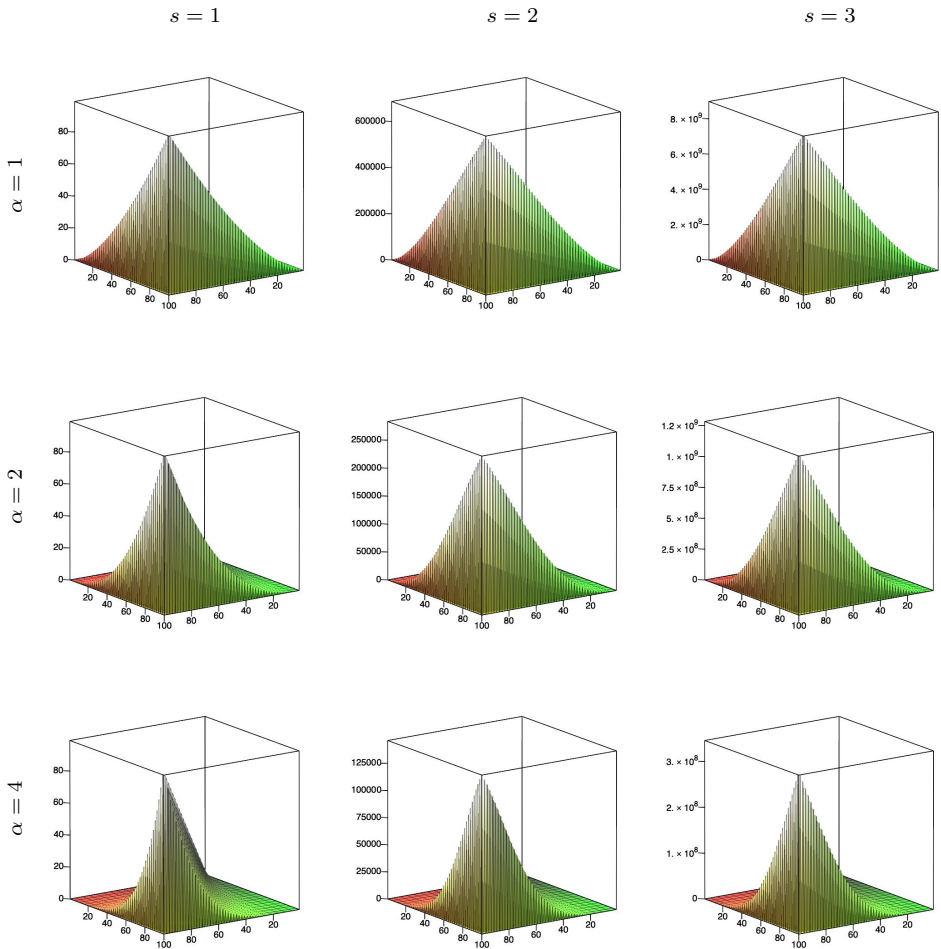

Fig.~\ref{fig:3.2} displays the magnitude of the powers of differentiation matrix for ultraspherical weights and different values of $\alpha$ and $s$, using the same rules of engagement as in Fig.~\ref{fig:3.1}. Two trends are discernible, both following from our discussion. Firstly, the decay along rows accelerates as $\alpha$ grows and $\DDD^{\,s}$ is more concentrated near the diagonal. Secondly, the terms along the diagonal (of course, with $m+n$ of the right parity) grow the fastest. Their rate of growth is rapid (and grows with $\alpha$) but this need not be a problem, at any rate once we approximate sufficiently smooth functions. In that case $\DDD$, its powers and possibly functions (e.g.\ $\exp(h\DDD)$) act on the expansion coefficients of functions in the underlying basis $\Phi$. Provided these functions are sufficiently smooth, it is plausible that these coefficients decay very rapidly and, for analytic functions, at an exponential rate. (We defer to Section~4 for more substantive discussion of convergence.) Thus,  large terms along the diagonal will multiply small terms in a vector of expansion  coefficients -- something that might conceivably cause loss of accuracy for truly huge matrices but which is probably negligible in practice.

Similarly to Laguerre weights, we now seek to prove that Theorem~4 provides also a sufficient condition for $\CC{ind}\,w_\alpha\geq s$ for ultraspherical weights, i.e.\ that $\alpha>s-1$ implies that $\DDD^{\,s}$ is bounded. Our method of proof is similar to that of Theorem~7, except that we need to account for a number of differences: firstly $\DDD_{m,n}$ can be nonzero only when $m+n$ is even, secondly, we have symmetric separability in place of separability and thirdly \R{eq:3.4} is no longer true and need be replaced by 
\begin{equation}
  \label{eq:3.10}
  \GG{a}_m\GG{b}_m=m+\alpha+\frac12.
\end{equation}

Letting again $k_0=m$, $k_s=n$, where $m\geq n+1$ (or $m\geq n$ once $s$ is even), we need to replace \R{eq:3.5} by 
\begin{displaymath}
  {\sum_{k_1=0}^\infty}^{\!\star} {\sum_{k_2=0}^\infty }^{\!\star} \cdots {\sum_{k_{s-1}=0}^\infty}^{\!\!\!\star}  |\DDD_{k_0,k_1}\DDD_{k_1,k_2}\cdots \DDD_{k_{s-2},k_{s-1}} \DDD_{k_{s-1},k_s}|<\infty,
\end{displaymath}
where the star means that we sum only over pairs $(k_{i-1},k_{i})$ such that $k_{i-1}+k_{i}$ is odd. Note that Proposition~6 remains true for the `starred sum' except that the constant (of which we care little) is different. 

We again commence with $\lA$- and $\rA$-pre-chains. Little changes for a $\rA$ chain, since the sum remains finite. The only possible challenge to boundedness may originate in a $\lA$ chain. We analyse a $\lA$-pre-chain using \R{eq:3.10},
\begin{Eqnarray*}
  &&{\sum_{k_\ell=k_{\ell-1}+1}^\infty}^{\hspace*{-10pt}\star}\hspace*{8pt}  {\sum_{k_{\ell+1}=k_{\ell}+1}^\infty}^{\hspace*{-10pt}\star} \hspace*{8pt}  \cdots {\sum_{k_{\ell+r-1}=k_{\ell+r-2}+1}^\infty}^{\hspace*{-22pt}\star}\hspace*{22pt} |\DDD_{k_\ell,k_{\ell-1}} \DDD_{k_{\ell+1},k_\ell} \cdots \DDD_{k_{\ell+r-1},k_{\ell+r-2}}|\\
  &=&{\sum_{k_\ell=k_{\ell-1}+1}^\infty}^{\hspace*{-10pt}\star}\hspace*{8pt}  {\sum_{k_{\ell+1}=k_{\ell}+1}^\infty}^{\hspace*{-10pt}\star} \hspace*{8pt}  \cdots {\sum_{k_{\ell+r-1}=k_{\ell+r-2}+1}^\infty}^{\hspace*{-22pt}\star}\hspace*{22pt} \GG{a}_{k_j}\GG{b}_{k_{j-1}}\\
  &=&\GG{b}_{k_{\ell-1}} {\sum_{k_\ell=k_{\ell-1}+1}^\infty}^{\hspace*{-10pt}\star}\hspace*{2pt} \left(k_\ell+\alpha+\frac12\right) {\sum_{k_{\ell+1}=k_{\ell}+1}^\infty}^{\hspace*{-10pt}\star} \hspace*{2pt} \left(k_{\ell+1}+\alpha+\frac12\right)  \cdots {\sum_{k_{\ell+r-1}=k_{\ell+r-2}+1}^\infty}^{\hspace*{-22pt}\star}\hspace*{18pt} \GG{a}_{k_{\ell+r-1}},
\end{Eqnarray*}
and, using Proposition~6,  transition seamlessly to a $\lA$-chain, while disregarding lower-order terms,
\begin{Eqnarray*}
  &&{\sum_{k_\ell=k_{\ell-1}+1}^\infty}^{\hspace*{-12pt}\star}\hspace*{2pt} \left(k_\ell+\alpha+\frac12\right) {\sum_{k_{\ell+1}=k_{\ell}+1}^\infty}^{\hspace*{-10pt}\star} \hspace*{2pt} \left(k_{\ell+1}+\alpha+\frac12\right)  \cdots {\sum_{k_{\ell+r-1}=k_{\ell+r-2}+1}^\infty}^{\hspace*{-22pt}\star}\hspace*{18pt} \GG{a}_{k_{\ell+r-1}}^2\\
  &=&{\sum_{k_{\ell+r-1}=k_{\ell-1}+r}^\infty}^{\hspace*{-20pt}\star} \GG{a}_{k_{\ell+r-1}}^2 {\sum_{k_\ell=k_{\ell-1}+r-1}^{k_{\ell+r-1}-r+1}}^{\hspace*{-0pt}\star} k_\ell  {\sum_{k_{\ell+1}=k_{\ell-1}+r-2}^{k_{\ell+r-1}-r+2}}^{\hspace*{-4pt}\star}\ k_{\ell+1} \cdots {\sum_{k_{\ell+r-2}=k_{\ell-1}+1}^{k_{\ell+r-1}-1}}^{\hspace*{-8pt}\star}\hspace*{6pt} k_{\ell+r-2}\\
  &\sim& {\sum_{k_{\ell+r-1}=k_{\ell-1}+r}^\infty}^{\hspace*{-20pt}\star} \GG{a}_{k_{\ell+r-1}}^2 \mathcal{Q}_{2r-2}(k_{\ell+r-1}) \sim {\sum_{k=k_{\ell-1}+r}^\infty}^{\hspace*{-10pt}\star} \hspace*{8pt}\frac{1}{k^{2\alpha-2r+1}}.
\end{Eqnarray*}
Consequently $\alpha>r$ is necessary and sufficient for convergence for each $\lA$-chain. Since the length of an $\lA$-chain is at most $s-1$, we deduce, similarly to Theorem~7 that

\begin{theorem}
  $\CC{ind}\,w_\alpha\geq s$ for the ultraspherical weight if an only if $\alpha>s-1$.
\end{theorem}

Both Theorems~7 and~8 present the same inequality. This is not surprising since, for both weights, $\alpha$ measures the `strength' of zero at the endpoint(s).

\subsection{Counterexamples: generalised Hermite and Konoplev\\ weights}

Ultraspherical and Laguerre weights are the obvious and most elementary choice in the intervals $(-1,1)$ and $(0,\infty)$ respectively and they are both separable in the sense of this paper. This might lead to an impression that separability is ubiquitous: this would be highly misleading. 

\begin{lemma}
  Let
  \begin{Eqnarray}
    \label{eq:3.11}
    \iota_{m,n}&=&\DDD_{m,n}\DDD_{m+1,n+1}-\DDD_{m+1,n}\DDD_{m,n+1},\\
    \label{eq:3.12}
    \check{\iota}_{m,n}&=&\DDD_{m,n}\DDD_{m+2,n+2}-\DDD_{m+2,n}\DDD_{m,n+2}.
  \end{Eqnarray}
  Separability implies that $\iota_{m,n}=0$ for all $m\geq n+2$, while symmetric separability implies that $\check{\iota}_{m,n}=0$ for all $m+n$ odd, $m\geq n+2$.
\end{lemma}

\begin{proof}
  Follows at once from the definition of (symmetric) separability.
\end{proof}

Note that neither \R{eq:3.11} nor \R{eq:3.12} are sufficient. Thus, a skew-symmetric $\DDD$ such that $\DDD_{2m+1,n}=0$ for all $m\in\BB{Z}_+$ and $n\leq 2m-1$ satisfies \R{eq:3.11} but in general is not separable. Likewise, a tridigonal skew-symmetric matrix obeys \R{eq:3.12} but is not symmetrically separable -- this is the case with the differentiation matrix associated with the Hermite weight, for example. Trying weights at random and computing, say, $\iota_{2,0}$ leads time and again to non-separable weights.

To explore further the (non)existence of separable weights, we examine two weights, generalisations of Hermite and ultraspherical weights respectively, but endowed with an additional parameter: the {\em generalised Hermite\/} and {\em Konoplev\/} weights. 

\subsubsection{Generalised Hermite weights}

Letting $\mu>-\frac12$, we examine the weight
\begin{equation}
  \label{eq:3.13}
  w_\mu(x)=|x|^{2\mu} \ee^{-x^2},\qquad x\in\BB{R} 
\end{equation}
\cite[p.~156]{chihara78iop}, originally considered by Szeg\H{o}.\footnote{We resist calling them ``Szeg\H{o} polynmials'' since the name is reserved for another type of polynomials, orthogonal in the complex unit circle  \cite{simon05opuc1,szego75op}.} It can be easily deduced from \cite[p.~156--7]{chihara78iop} that the underlying W-functions are
\begin{Eqnarray*}
  \varphi_{2n}(x)&=&(-1)^n 2^n \sqrt{\frac{n!}{\CC{\Gamma}(n+\mu+\frac12)}} \CC{L}_n^{(\mu-\frac12)}(x^2)|x|^\mu \ee^{-x^2/2},\\
  \varphi_{2n+1}(x)&=&(-1)^n 2^{n+1}\sqrt{\frac{n!}{2\CC{\Gamma}(n+\mu+\frac32)}} x\CC{L}_n^{(\mu+\frac12)}(x^2)|x|^\mu \ee^{-x^2/2},\qquad n\in\BB{Z}_+.
\end{Eqnarray*}

Generalised Hermite weights are of marginal importance to the work of this paper and although their differentiation matrix can be derived explicitly,
\begin{Eqnarray*}
  \DDD_{2n,2n-1}&=&\sqrt{n},\qquad \DDD_{2m,2n-1}=0,\quad m\geq n+1,\\
  \DDD_{2n+1,2n}&=&\frac{2n+1}{2\sqrt{n+\mu+\frac12}},\quad \DDD_{2m+1,2n}=(-1)^{m+n-1}\mu \sqrt{\frac{m!}{(n+\mu+\frac12)_{m+1-n}}},
\end{Eqnarray*}
with skew-symmetric complement, we will not present here a formal (and lengthy) algebra. Instead, a reader might use a symbolic algebra package to compute the first few elements, enough to evaluate $\iota_{2,0}$ and $\check{\iota}_{3,0}$ and check that they are both nonzero -- in light of Lemma~9 this is sufficient to rule out separability and symmetric separability, respectively. 

As a matter of fact, $\DDD$ has an interesting shape: its $(2m+1)$st columns (hence also the $(2n+1)$st rows) are consistent with a tridiagonal matrix, more specifically with the differentiation matrix corresponding to the standard Hermite weight (i.e.\ with $\mu=0$). More specifically, 
\begin{displaymath}
  \iota_{2n,2n-1}=\left(n+\frac12\right)\sqrt{\frac{n}{n+\mu+\frac12}}\neq0,\qquad \iota_{2n+1,2n}=\left(n+\frac12\right)\sqrt{\frac{n+1}{n+\mu+\frac12}}\neq0,
\end{displaymath}
otherwise $\iota_{m,n}=0$ for $m\geq n+2$, while
\begin{Eqnarray*}
  \check{\iota}_{2n+3,2n}&=&\frac{\mu\sqrt{(n+2)!} [\mu\sqrt{(n+1)!}+n+\frac32]}{(n+\mu+\frac32)\sqrt{(n+\mu+\frac12)(n+\mu+\frac52)}} \neq 0,\\
  \check{\iota}_{2m+4,2n+1}&=&0,\qquad m\geq n+2
\end{Eqnarray*} 
In each case the separability tests \R{eq:3.11} and \R{eq:3.12} fail only marginally -- but fail nonetheless.

\subsubsection{Konoplev weights}

Letting $\alpha,\gamma>-1$, we set
\begin{equation}
  \label{eq:3.14}
  w_{\alpha,\gamma}(x)=|x|^{2\gamma+1}(1-x^2)^\alpha,\qquad x\in(-1,1).
\end{equation}
The weight \R{eq:3.14}, which has been considered in \cite{konoplev61por,konoplev65abo} and described in \cite[p.~155]{chihara78iop}, generalises ultraspherical weights by adding the possible weakly singular factor $|x|^{2\gamma+1}$. Specifically, $w_{\alpha,\gamma}\in\CC{C}^s(-1,1)$ if and only if
\begin{displaymath}
  \mbox{either}\qquad \gamma\in\left\{ \frac{k}{2}\,:\, k\in\{-1,0,\ldots,s-2\}\right\}\qquad\mbox{or}\qquad \gamma>\frac{s-1}{2}.
\end{displaymath}
The underlying orthogonal polynomial system is
\begin{displaymath}
  \CC{S}_{2n}(x)=\CC{P}_n^{(\alpha,\gamma)}(2x^2-1),\qquad \CC{S}_{2n+1}(x)=x\CC{P}_n^{(\alpha,\gamma+1)}(2x^2-1),\qquad n\in\BB{Z}_+,
\end{displaymath}
and the  monic polynomials obey the three-term recurrence relation
\begin{displaymath}
  \hat{\CC{S}}_{n+1}(x)=x\hat{\CC{S}}_n(x)-c_n \hat{\CC{S}}_{n-1}(x),
\end{displaymath}
where
\begin{displaymath}
  c_{2n}=\frac{n(n+\alpha)}{(2n+\alpha+\gamma)(2n+1+\alpha+\gamma)},\qquad c_{2n+1}=\frac{(n+1+\gamma)(n+1+\alpha+\gamma)}{(2n+\alpha+\gamma)(2n+1+\alpha+\gamma)}.
\end{displaymath}
Replacing Jacobi polynomials by their orthonormal counterparts and using a formula from \cite[p.~260]{rainville60sf}, easy algebra confirms that
\begin{Eqnarray*}
  \kappa_{2m}^{\alpha,\gamma}&=&\int_{-1}^1 w_{\alpha,\gamma}(x) S_{2m}^2(x)\D x=\frac{(m+1+\alpha+\gamma)_m \G(m+1+\alpha)\G(m+1+\gamma)}{m!\G(2m+2+\alpha+\gamma)},\\
  \kappa_{2m+1}^{\alpha,\gamma}&=&\int_{-1}^1 w_{\alpha,\gamma}(x) S_{2m+1}^2(x)\D x=\frac{(m+2+\alpha+\gamma)_m\G(m+1+\alpha)\G(m+2+\gamma)}{m!\G(2m+3+\alpha+\gamma)},
\end{Eqnarray*}
therefore
\begin{Eqnarray*}
  \varphi_{2m}(x)&=& \frac{|x|^{\gamma+\frac12}(1-x^2)^{\alpha/2}}{\sqrt{\kappa_{2m}^{\alpha,\gamma}}} \CC{P}_m^{(\alpha,\gamma)}(2x^2-1),\\
  \varphi_{2m+1}(x)&=&\frac{x|x|^{\gamma+\frac12}(1-x^2)^{\alpha/2}}{\sqrt{\kappa_{2m+1}^{\alpha,\gamma}}}\CC{P}_m^{(\alpha,\gamma+1)}(2x^2-1).
\end{Eqnarray*}
The weights \R{eq:3.14} are symmetric, thus we examine the possibility of symmetric separability. A brute-force computation yields
\begin{displaymath}
  \check{\iota}_{3,0}=(5+\alpha+\gamma)(2\gamma+1) \sqrt{\frac{(4+\alpha+\gamma)(6+\alpha+\gamma)}{2(1+\alpha)(2+\alpha)(1+\gamma)(3+\gamma)}},
\end{displaymath}
ruling out symmetric separability except for the case $\gamma=-\frac12$, which corresponds to the ultraspherical weight. 

\subsubsection{A limiting behaviour of the $\iota_{m,n}$s}

While separability, hence $\iota_{m,n}=0$ for $m\geq n+2$, appears to be exceedingly rare, we claim that the latter holds more broadly in a much weaker, asymptotic form.

Let $w$ be a weight in $(a,b)$, $w(a)=w(b)=0$, with the underlying orthonormal polynomials $\{p_n\}_{n=0}^\infty$, where the coefficient of $x^n$ in $p_n$ is $k_n>0$. Comparing the coefficients of $x^{n+1}$ in the three-term recurrence relation \R{eq:1.3} we deduce at once that $k_{n+1}/k_n=\beta_n^{-1}$.

\begin{theorem}
  Supposing that $\beta_n\geq \beta^*>0$ and $\tilde{w}(x)=[w'(x)]^2/w(x)$ is itself a weight function in $(a,b)$, it is true that
  \begin{equation}
    \label{eq:3.15}
    \lim_{m\rightarrow\infty}\iota_{m,n}=0,\qquad n\in\BB{Z}_+.
  \end{equation}
\end{theorem}

\begin{proof}
  Letting $m\geq n+2$, \R{eq:2.3} yields
  \begin{Eqnarray*}
  \iota_{m,n}&=&\frac14 \int_a^b w'(y) p_m(y)p_n(y)\D y \int_a^b w'(x) p_{m+1}(x) p_{n+1}(x)\D x\\
  &&\hspace*{20pt}\mbox{}-\frac14 \int_a^b w'(x) p_{m+1}(x)p_n(x)\D x\int_a^b w'(y) p_m(y) p_{n+1}(y)\D y\\
  &=&\frac14 \int_a^b \int_a^b w'(x)w'(y) p_{m+1}(x)p_m(y)[ p_{n+1}(x)p_n(y)-p_n(x)p_{n+1}(y)]\D x\D x.
\end{Eqnarray*}
We recall the {\em Christoffel--Darboux formula\/},
\begin{displaymath}
  \sum_{\ell=0}^n p_\ell(x)p_\ell(y)=\frac{k_n}{k_{n+1}} \frac{p_{n+1}(x)p_n(y)-p_n(x)p_{n+1}(y)}{x-y},
\end{displaymath}
where $k_n>0$ is the coefficient of $x^n$ in $p_n$ \cite[p.~153]{chihara78iop}. Therefore 
\begin{Eqnarray}
  \nonumber
  \iota_{m,n}&=&\frac{k_{n+1}}{k_n}\int_a^b \int_a^b w'(x)w'(y) p_{m+1}(x) p_m(y) (x-y) \sum_{\ell=0}^n p_\ell(x)p_\ell(y) \D x\D y\\
  \label{eq:3.16}
  &\leq&\frac{1}{\beta^*} \left|\int_a^b \int_a^b w'(x)w'(y) p_{m+1}(x) p_m(y) (x-y) \sum_{\ell=0}^n p_\ell(x)p_\ell(y) \D x\D y\,\right|\!,
\end{Eqnarray}
because $\beta_n^{-1}\leq {\beta^*}^{-1}$. Letting $n\rightarrow\infty$ in \R{eq:3.16}, we obtain
\begin{Eqnarray*}
  \lim_{n\rightarrow\infty}\iota_{m,n}&\leq& \frac{1}{\beta^*}\left|\int_a^b \int_a^b w'(x)w'(y)p_{m+1}(x)p_m(x) (x-y)  \sum_{\ell=0}^\infty p_\ell(x)p_\ell(y) \D x\D y\right|\\
  &=&\frac{1}{\beta^*}\left|\int_a^b \int_a^b \frac{w'(x)w'(y)}{\sqrt{w'x)w(y)}} p_{m+1}(x)p_m(x) (x-y) K(x,y)\D x\D y\right|,
\end{Eqnarray*}
where $K$ is tha Christoffel--Darbeaux kernel from the proof of Theorem~5. According to \R{eq:2.8} it is a reproducing kernel and it follows at once that the double integral vanishes.
\end{proof}

The condition $\beta_n>\beta^*$, $n\in\BB{Z}_+$, is very weak: we already know that $\beta_n>0$, all the condition says is that, in addition, the $\beta_n$s are bounded away from zero.

\setcounter{equation}{0}
\setcounter{figure}{0}
\section{Computational aspects}

\subsection{A product of $\DDD$ by a vector}

Consider a separable weight function, e.g.\ a Laguerre weight. A major task in practical implementation of the ideas of this paper to spectral methods is to form a product of the form $\MM{h}=\DDD\,\MM{f}$, where $\MM{f}$ is a (real or complex) infinite-dimensional vector. In most applications $\MM{f}$ is likely to be the vector of expansion coefficients of a function $f$ in the basis $\Phi$, which is likely to decay rapidly. For example, if $f$ is analytic in an ellipse enclosing $(a,b)$, we expect the $|f_k|$s to decay at an exponential rate. 

Forming $\MM{h}$ and, with greater generality, products of the form $\DDD^{\,r}\MM{f}$ for $r\in\BB{N}$, is important for the obvious reason that approximating partial differential equations is likely to entail forming derivatives. Perhaps less trivial reason is the formation of matrix functions. Thus, let $g$ be a (typically analytic) function and we wish to form $g(\DDD)\MM{f}$: standard examples are the exponential $g(z)=\ee^z$, the cosine $g(z)=\cos z$ and the sinc function $g(z)=z^{-1}\sin z$. Two of the most popular methods are the Krylov subspace algorithm and a quadrature of an integral representation of $g$ \cite{higham08fom}, and both require the formation of products of the kind considered in this subsection. 

Tridiagonal differentiation matrices, of the form considered in \cite{iserles19oss}, seemingly enjoy strong advantage in this context, because the formation of $\DDD_M \MM{f}_{\!M}$, where $\MM{f}\in\BB{R}^{M+1}$, entails just $\approx 3M$ floating-point operations (flops).\footnote{Products of the form $\mbox{\smallcurly{D}}_M^{\,r}\Mm{f}_{\!M}$ can thus be formed  consecutively in $\approx 3rM$ flops.} We demonstrate that this `fast product' property is shared by differentiation matrices formed by separable or symmetrically separable weights. Actually, the structure of these matrices allows for infinite-dimensional computation: the starting point is an integer $N$, typically much larger than $M$, such that $|f_m|$ is negligible (in practical terms, smaller than a user-provided error tolerance) for $m>N$, and we wish to form
\begin{equation}
  \label{eq:4.1}
  h_m=\sum_{n=0}^N \DDD_{m,n} f_n,\qquad m=0,\ldots,M.
\end{equation}
We commence by assuming that a weight is separable, whereby \R{eq:3.1} yields
\begin{displaymath}
  h_m=- \GG{a}_m \sum_{n=0}^{m-1} \GG{b}_n f_n + \GG{b}_m\!\! \sum_{n=m+1}^N \GG{a}_n f_n=\sigma_m+\rho_m,\qquad m=0,\ldots,M,
\end{displaymath}
where
\begin{displaymath}
  \sigma_m=\sum_{n=0}^{m-1} \GG{b}_n f_n,\quad \rho_m=\sum_{n=m+1}^N \GG{a}_n f_n,\qquad m=0,\ldots,M.
\end{displaymath}
Then
\begin{Eqnarray*}
  h_0&=&\GG{b}_0 \rho_0,\\
  h_m&=&-\GG{a}_m\sigma_m+\GG{b}_m\rho_m,\qquad m=1,\ldots,M,\\
  &&\mbox{where}\quad \sigma_m=\sigma_{m-1}+\GG{b}_{m-1}f_{m-1},\quad \rho_m=\rho_{m-1}-\GG{a}_m f_m.
\end{Eqnarray*}
Assuming that the $\GG{a}_m$s and $\GG{b}_n$s have been precomputed (and this need be done only once, no matter how many products are required), the calculation \R{eq:4.1}  takes just $\approx N+4M$ flops -- and by the same token, computing the first $M+1$ entries of $\DDD_N^{\,r}\!\MM{f}_{\!N}$ takes $\approx r(N+4M)$ flops. 

Similar operations count applies to symmetrically separable weight, whereby the entires of $\DDD$ obey \R{eq:3.2}. Assuming that both $M$ and $N$ are even, we have 
\begin{Eqnarray*}
  h_{2m}&=&\sum_{n=0}^{N/2} \DDD_{2m,2n+1} f_{2n+1},\\
  h_{2m+1}&=&\sum_{n=0}^{N/2} \DDD_{2m+1,2n} f_{2n},\qquad m=0,\ldots,\frac{M}{2}.
\end{Eqnarray*}
Therefore
\begin{Eqnarray*}
  h_{2m}&=&\GG{a}_{2m}\sum_{n=0}^{m-1} \GG{b}_{2n+1} f_{2n+1}-\GG{b}_{2m}\sum_{n=m+1}^{N/2} \GG{a}_{2n+1} f_{2n+1},\\
  h_{2m+1}&=&\GG{a}_{2m+1} \sum_{n=0}^{m-1} \GG{b}_{2n} f_{2n}-\GG{b}_{2m+1}\sum_{n=m+1}^{N/2} \GG{a}_{2n} f_{2n},\qquad m=0,\ldots,\frac{M}{2}.
\end{Eqnarray*}
Set
\begin{Eqnarray*}
  \sigma_m^{\CC{E}}&\displaystyle =\sum_{n=0}^{m-1} \GG{b}_{2n} f_{2n},\qquad &\sigma_m^{\CC{O}}=\sum_{n=0}^{m-1} \GG{b}_{2n+1}f_{2n+1},\\
  \rho_m^{\CC{E}}&\displaystyle =\sum_{n=m+1}^{N/2} \GG{a}_{2n} f_{2n},\qquad &\rho_m^{\CC{O}}=\sum_{n=m+1}^{N/2} \GG{a}_{2n+1} f_{2n+1},
\end{Eqnarray*}
hence
\begin{displaymath}
  h_{2m}=\GG{a}_{2m}\sigma_m^{\CC{O}}-\GG{b}_{2m} \rho_m^{\CC{O}},\qquad h_{2m+1}=\GG{a}_{2m+1} \sigma_m^{\CC{E}}-\GG{b}_{2m+1} \rho_m ^{\CC{E}},\qquad m=0,\ldots,\frac{M}{2}.
\end{displaymath}
However,
\begin{displaymath}
  \sigma_0^{\CC{E}}=\sigma_0^{\CC{O}}=0,\qquad \rho_0^{\CC{E}} =\sum_{m=1}^{N/2} \GG{a}_{2n}f_{2n},\qquad \rho_0^{\CC{O}}=\sum_{m=1}^{N/2} \GG{a}_{2n+1} f_{2n+1},
\end{displaymath}
and\\[4pt]
\begin{tabular}{lcl}
  $\sigma_m^{\CC{E}}=\sigma_{m-1}^{\CC{E}}+\GG{b}_{2m-2}f_{2m-2},$ &\qquad & $\sigma_m^{\CC{O}}=\sigma_{m-1}^{\CC{O}}+\GG{b}_{2m-1} f_{2m-1}$,\\[4pt]
  $\rho_m^{\CC{E}}=\rho_{m-1}^{\CC{E}}-\GG{a}_{2m}f_{2m},$&\qquad & $\displaystyle \rho_m^{\CC{O}}=\rho_{m-1}^{\CC{O}}-\GG{a}_{2m+1}f_{2m+1},\qquad m=1,\ldots,\frac{M}{2}.$
\end{tabular}\\[4pt]
Thus, again, we need just $\approx 4M+N$ flops to compute the first $M+1$ terms of $\DDD_N\MM{f}_{\!N}$. 

\subsection{Speed of convergence}

While the convergence of orthogonal polynomials to `nice' (in particular, analytic) functions is well understood, this is not the case for W-functions.  Comprehensive analysis of their convergence and its speed is a matter for future work, yet it is of an interest to present preliminary computational results, not least as a preamble for a discussion on the choice of the optimal value of $\alpha$ in the context of either Laguerre or ultraspherical weights, while comparing them to standard approximation by the underlying orthogonal polynomials.

It rapidly becomes apparent that we have a competition between different imperatives:
\begin{itemize}
\item {\em The number of zero boundary conditions:\/} This determines the value of $\alpha$ and, according to Theorems~7 and 8, we need $\alpha>s-1$ in $\HHH{s}(a,b)$.
\item {\em Regularity of approximating functions:\/} While $\PPP$  consists of polynomials, hence analytic functions, this is not the case with $\Phi$, whether in the context of ultraspherical or Laguerre weights: it all depends on the value of $\alpha$. If $\alpha$ is an even integer then the $\varphi_n$s are analytic, otherwise analyticity fails at the endpoints.
\item {\em The underlying function space:\/} Much depends on how the error is measured. Among the many possibilities, we single out two: the $\HHH{p}(a,b)$ norm for a suitable value of $p$ (in particular, the $\CC{L}_2(a,b)$ norm) and the $\CC{L}_\infty[a,b]$ (and, more generally, $\CC{H}_\infty^p[a,b]$) norm. The choice of a norm depends on the underlying application.
\end{itemize}
As things stand, we cannot report any `hard' results. Yet, even preliminary numerical experimentation reveals a remarkable state of affairs.

In the following figures we let $\alpha$ be in $\{1,2,3,4\}$. In this and all  figures in this paper we denote $\alpha=1$ by a red, dotted line, $\alpha=2$ by a magenta solid line, $\alpha=3$ by a green dashed line and, finally, $\alpha=4$ by a blue dash-dotted line. Because of the rapid decay of errors, we display them all in a logarithmic scale to base 10 -- in other words, the $y$-axis displays the number of decimal digits. Given a function $f$ and recalling the expansion coefficients $\hat{f}_n^P$ and $\hat{f}^\Phi_n$ from Remark~1, corresponding to  expansions in $\PPP$ and $\Phi$ respectively, we let
\begin{displaymath}
  F^P_N(x)=\sum_{n=0}^N \hat{f}_n^P p_n(x),\quad F^\Phi_N(x)=\sum_{n=0}^N \hat{f}_n^\Phi \varphi_n(x),\qquad N\in\BB{Z}_+.
\end{displaymath}
Thus, $F^P_N-f$ and $F^\Phi_N-f$ are the (pointwise) errors with respect to the polynomial and the W-function basis, respectively, and we need to measure them in an appropriate norm. We denote by ${}^d\!F_N^P$  the derivative expansion, i.e.\ with $p_n$ and $f$ replaced by $p_n'$ and $f'$ respectively,  similarly for higher derivatives and for $F^\Phi_N$.

\subsubsection{Ultraspherical W-functions}

We commence from ultraspherical weights and consider
\begin{equation}
  \label{eq:4.2}
  f(x)=(1-2x)\cos\frac{\pi x}{2}\in\HHH{1}(-1,1).
\end{equation}

\begin{figure}[htb]
  \begin{center}
    \includegraphics[width=180pt]{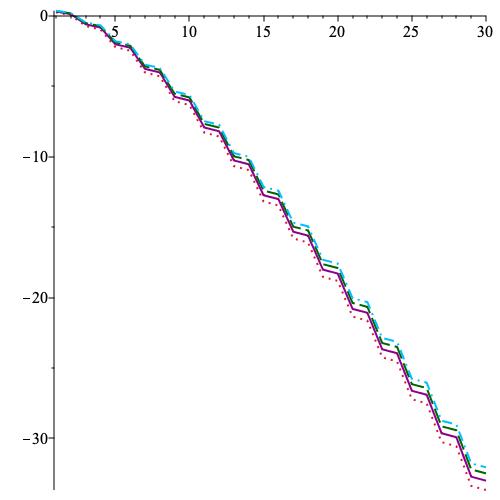}\includegraphics[width=180pt]{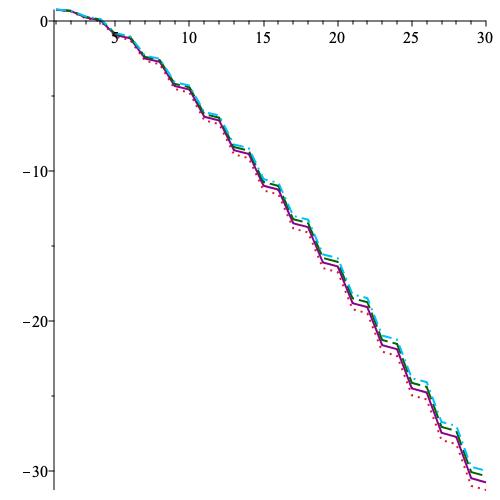}
    \includegraphics[width=180pt]{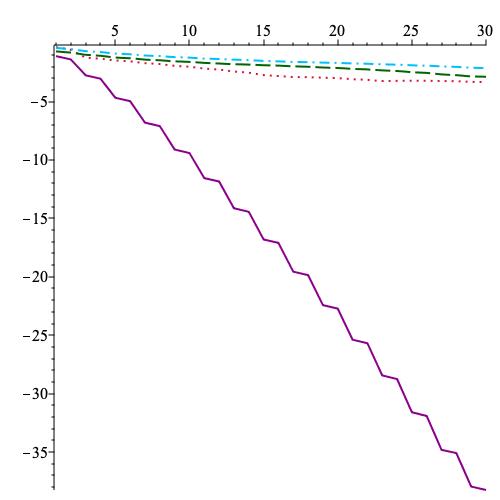}\includegraphics[width=180pt]{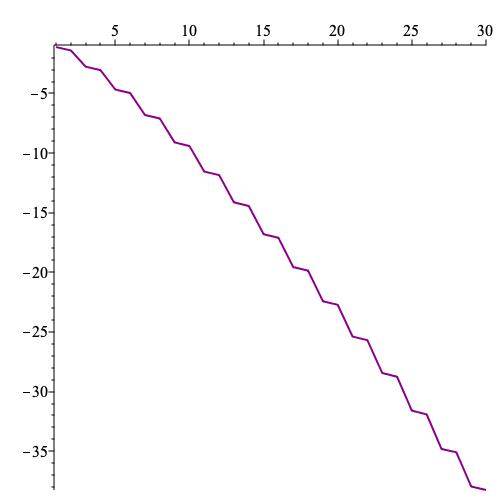}
    \caption{Ultraspherical W-functions: The errors $\log_{10}\|F_N^P-f\|_\infty$ (top left), $\log_{10}\|{}^d\!F_N^P-f'\|_\infty$ (top right), $\log_{10}\|F_N^\Phi-f\|_\infty$ (bottom left) and $\log_{10}\|{}^d\!F_N^\Phi-f'\|_\infty$ (bottom right) for $\alpha=1,2,3,4$ and $N=1,2,\ldots,30$ (except that in the bottom-right plot only $\alpha=2$ is displayed).}
    \label{fig:4.1}
  \end{center}
\end{figure}

In Fig.~\ref{fig:4.1} we display in logarithmic norm (in other words, the $y$-axis displays the number of significant digits of the error) the $\CC{L}_\infty[-1,1]$ error for polynomial approximation to $f$ and its first derivative (top row) and for W-functions for the ultraspherical weight.\footnote{Polynomial approximation, of course, leads to an unstable spectral method. Yet, its error and its comparison with the error committed by W-functions are of an independent interest.} Polynomial approximation -- as can be expected from general theory and the analyticity of $f$ -- decays at an exponential speed and, for $N=30$, we attain $\approx 32$ significant digits. This is also the case with derivatives, with a very minor degradation in accuracy. The error for W-functions, though, is radically different. The errors for $\alpha\in\{1,3,4\}$ decay very slowly, at a polynomial rate, and for $N=30$ we recover just $\approx 4$ significant digits, an unacceptably large error. On the other hand, the error for $\alpha=2$ at $N=30$ is $\approx 3\times 10^{-39}$, significantly better than polynomial approximation! 

The reason for this miraculous behaviour for $\alpha=2$ bears some attention. Little surprise perhaps that $\alpha=1$ behaves poorly because it is at the wrong end of the boundedness condition for $\DDD^{\,2}$. However, as a matter of fact, we do not consider second derivatives in this particular instance and $\alpha\in\{3,4\}$ are just as bad.   The reasons are as follows. For $\alpha\in\{1,3\}$ the $\varphi_n$s have a weak singularity along the boundary, while $\varphi_n'$ becomes singular there. For $\alpha=4$, on the other hand, $\varphi_n'(\pm1)=0$ mean that $\CC{L}_\infty$ convergence of derivatives is impossible unless also the derivatives of $f$ vanish at the endpoints. (This is the reason why $\log_{10}\|{}^d\!F_N^\Phi-f'\|_\infty$ is  displayed only for $\alpha=2$.)

\begin{figure}[htb]
  \begin{center}
    \includegraphics[width=180pt]{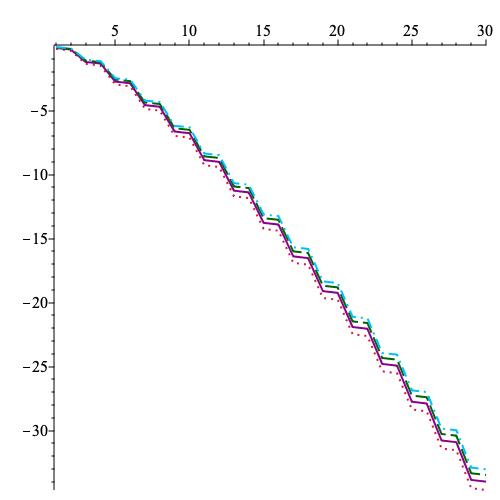}\includegraphics[width=180pt]{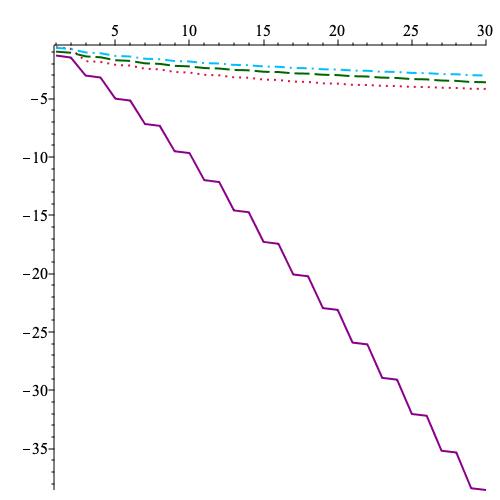}
    \caption{Ultraspherical W-functions: The errors $\log_{10}\|F_N^P-f\|_2$ (left) and $\log_{10}\|F_N^\Phi-f\|_2$ (right) for $\alpha=1,2,3,4$.}
    \label{fig:4.2}
  \end{center}
\end{figure}

Not much changes if, instead of $\CC{L}_\infty$, we compute an $\HHH{p}$ error, except that in general $\CC{L}_2$-like norms are more forgiving. In principle, neither singularities or excessive vanishing of derivatives at the endpoints need prevent convergence. Thus, in Fig.~\ref{fig:4.2} we plot the $\CC{L}_2(-1,1)$ errors for example~\ref{eq:4.2}. The overall picture remains the same: polynomial approximation decays at exponential rate and we attain, regardless of the choice of $\alpha$, about 34 significant digits for $N=30$, while W-function approximation for $\alpha\in\{1,3,4\}$ is very poor yet, for $\alpha=2$, we again hit the `sweet spot' and recover $\approx38$ significant digits. W-functions are vastly superior for $\alpha=2$, fail dismally otherwise. 

To explore further the error committed by ultraspherical W-functions we consider
\begin{equation}
  \label{eq:4.3}
  f(x)=(1-2x)\cos^2\frac{\pi x}{2}
\end{equation}
the only difference in this (not very imaginative!) choice is that now $f(\pm1)=f'(\pm1)=0$. We display the $\CC{L}_\infty$ error for $f^{(i)}$, $i=0,1,2$, in Fig.~\ref{fig:4.3} for the W-functions. The error in polynomial approximation is roughly independent of $\alpha$ and for $N=30$ we attain $\approx 24$ decimal digits for $f$, $\approx21$ for $f'$ and $\approx19$ for $f''$. By this stage we should not be surprised that $\alpha=1$ and $\alpha=3$ do badly in approximating $f$ because of the weak singularity at the endpoints and they fail altogether approximating derivatives. For $\alpha\in\{2,4\}$ the endpoints are analytic and indeed the underlying functions do very well indeed, definitely better than polynomial approximation. $\alpha=4$ is a winner, unsurprisingly because $f\in\HHH{2}(-1,1)$ and this is matched by $\Phi$. However, $\alpha=2$ does quite well, worse by perhaps two decimal digits but  still beating polynomial approximation. The reason is that too few zero Dirichlet boundary conditions do not prevent $\CC{L}_\infty$ convergence of an orthogonal sequence, although they might slow it up to a modest extent. On the other hand, excessive zero Dirichlet boundary conditions prevent $\CC{L}_\infty$ convergence at the endpoints. Thus, {\em the interplay between the number of zero boundary conditions and the choice of $\alpha$ is not symmetric!\/} It is always better to err by choosing smaller $\alpha$, as long as it is an even integer, consistent with the bound of Theorem~8.

\begin{figure}[htb]
  \begin{center}
    \includegraphics[width=125pt]{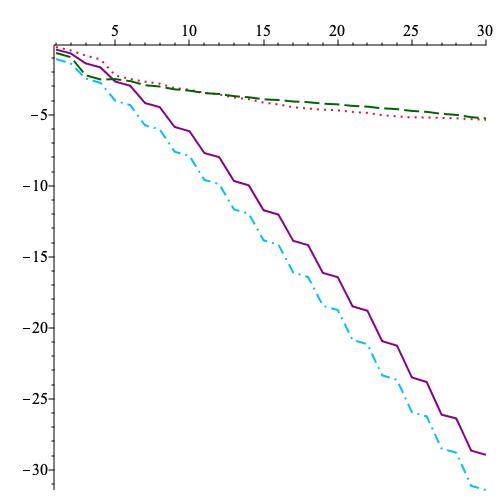}\includegraphics[width=125pt]{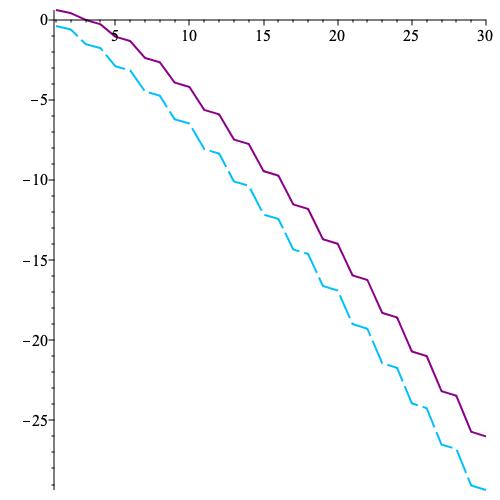}\includegraphics[width=125pt]{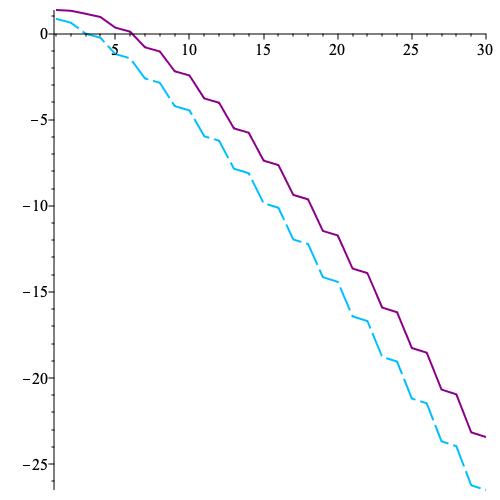}
    \caption{Ultraspherical W-functions: The errors $\log_{10}\|F_N^\Phi-f\|_2$ (left) $\log_{10}\|{}^d\!F_N^\Phi-f'\|_2$ (centre) and $\log_{10}\|{}^{dd}\!F_N^\Phi-f''\|_2$ (right) for $\alpha=1,2,3,4$ and the function \R{eq:4.3}.}
    \label{fig:4.3}
  \end{center}
\end{figure}

\subsubsection{Laguerre W-functions}

We are now concerned with the Laguerre weight and choose the model problem
\begin{equation}
  \label{eq:4.4}
  f(x)=\ee^{-x}\sin x,\qquad x\geq0.
\end{equation}
Note that $f(0)=0$, $f'(0)\neq0$. 

\begin{figure}[htb]
  \begin{center}
    \includegraphics[width=180pt]{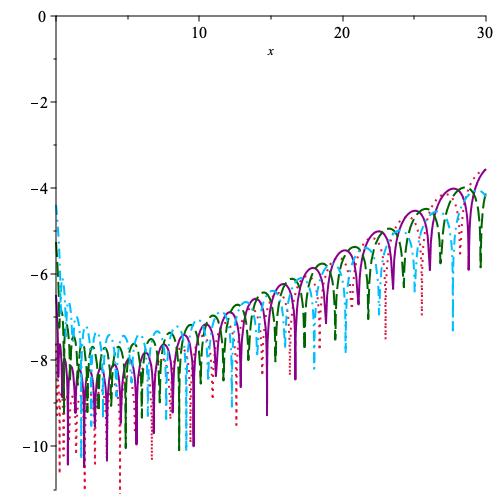}\includegraphics[width=180pt]{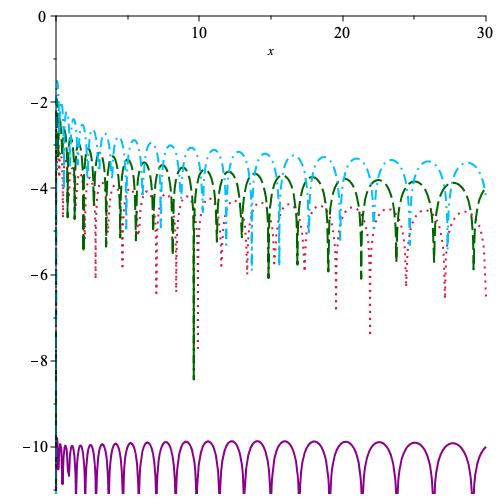}
    \caption{Laguerre W-functions: The errors $\log_{10}|F_{40}^P(x)-f(x)|$ (left) and $\log_{10}|F_{40}^\Phi(x)-f(x)|$ (right) for $x\in[0,30]$ and $\alpha=1,2,3,4$.}
    \label{fig:4.4}
  \end{center}
\end{figure}

An expansion in Laguerre (or any other) polynomials cannot be bounded in an infinite interval hence, instead of plotting $\log_{10}\|F_N^P-f\|_2$ for increasing values of $N$, we choose $N=40$ and plot the pointwise error in the interval $[0,30]$. This is evident on the left of Fig.~\ref{fig:4.4}: the error is just about fine for small $x>0$, subsequently growing rapidly (as a matter of fact, exponentially). On the other hand, as can be seen on the right of that figure, the error of W-functions is uniformly bounded. For $\alpha\in\{1,3,4\}$ it is fairly similar -- and unacceptably large -- while for $\alpha=2$ we attain $\approx10$ decimal digits of accuracy, apparently uniformly in $[0,\infty)$. Yet again we have the `sweet spot' for $\alpha=2$. This state of affairs remains true for the first few derivatives and the deterioration in accuracy using W-functions is very mild indeed.

Finally, we consider
\begin{equation}
  \label{eq:4.5}
  f(x)=\ee^{-x}\sin^2x,\qquad x\geq0.
\end{equation}
Now $f(0)=f'(0)=0$ and $f''(0)\neq0$. There is no need to display the $\CC{L}_\infty[0,\infty)$ error committed by Laguerre polynomials since, again, it is unbounded. 

\begin{figure}[htb]
  \begin{center}
    \includegraphics[width=180pt]{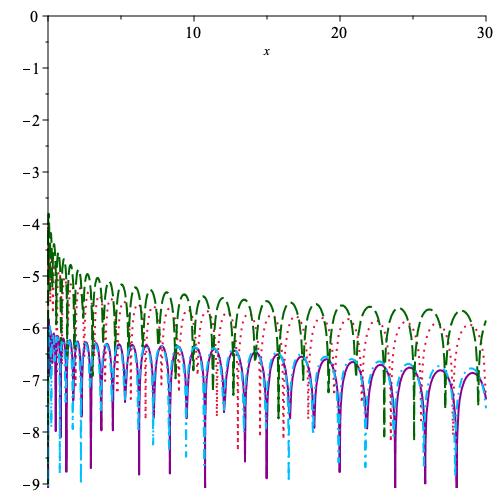}\includegraphics[width=180pt]{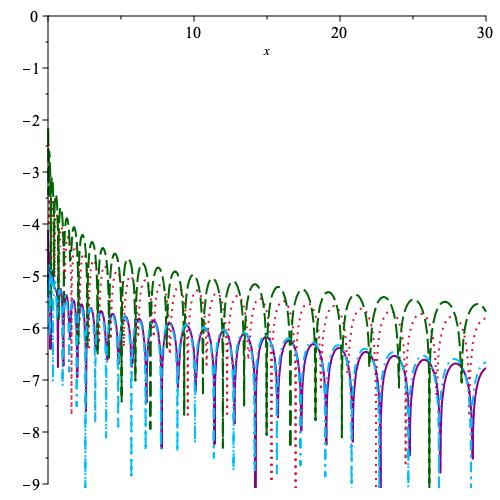}
    
    \includegraphics[width=180pt]{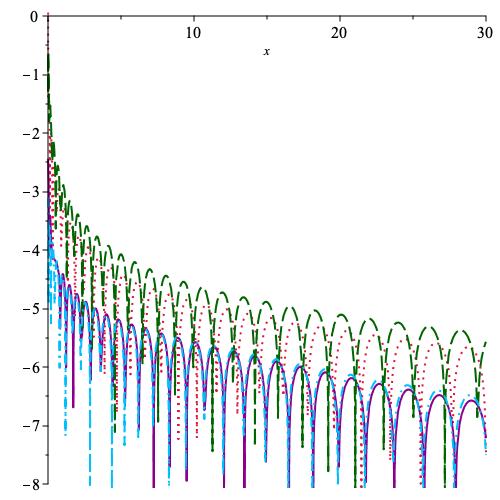}
    \caption{Laguerre W-functions: The errors $\log_{10}|F_{60}^\Phi(x)-f(x)|$ (left) $\log_{10}|{}^d\!F_{60}^\Phi(x)-f'(x)|$ (right) and $\log_{10}|{}^{dd}\!F_{60}^\Phi(x)-f''(x)|$ for the function \R{eq:4.5}, $x\in[0,30]$ and $\alpha=1,2,3,4$.}
    \label{fig:4.5}
  \end{center}
\end{figure}

In Fig.~\ref{fig:4.5} we employ the same colour and style scheme to plot the errors committed in $[0,30]$ for $f$, $f'$ and $f''$. Clearly, $\alpha=2$ and $\alpha=4$, the two values associated with analyticity at the origin, win insofar as approximating the function itself is concerned, although the margin is somewhat smaller than in our other examples. The approximation of the first and the second derivatives is more interesting: on the face of it, it is a dead heat between $\alpha=2$ and $\alpha=4$, but closer examination of the behaviour near the left endpoint unravels a crucial difference. For example. for $\eta=10^{-10}$ we have (to four significant digits)\\[6pt]
\begin{tabular}{c|cccc}
  $\alpha$ & 1 & 2 & 3 & 4\\\hline
  $|F_{60}^\Phi(\eta)-f(\eta)|$ & $2.128_{-08}$ & $9.631_{-15}$ & $1.434_{-16}$ & $7.778_{-24}$\\
  $|{}^d\!F_{60}^\Phi(\eta)-f'(\eta)|$ & $1.064_{+02}$ & $9.631_{-05}$ & $2.151_{-06}$ & $9.555_{-14}$\\
  $|{}^{dd}\!F^\Phi_{60}(\eta)-f''(\eta)|$ & $5.319_{+11}$ & $4.092_{-03}$ & $1.075_{+05}$ & $9.555_{-04}$
\end{tabular}\\[6pt]
The conclusion is clear. Once the inequality of Theorem~7 is breached, the approximation blows up at the origin: this happens with $\alpha=1$ and {\em any\/} derivative. The error for $\alpha=3$ decays for $N\gg1$ for the function value and the first derivative, but it blows up for the second derivative, while for $\alpha=2$ the progression to the correct boundary condition is considerably slower than for $\alpha=4$. This is apparent from Fig.~\ref{fig:4.6}: $\alpha=4$ wins, although by a small margin.

\begin{figure}[htb]
  \begin{center}
    \includegraphics[width=180pt]{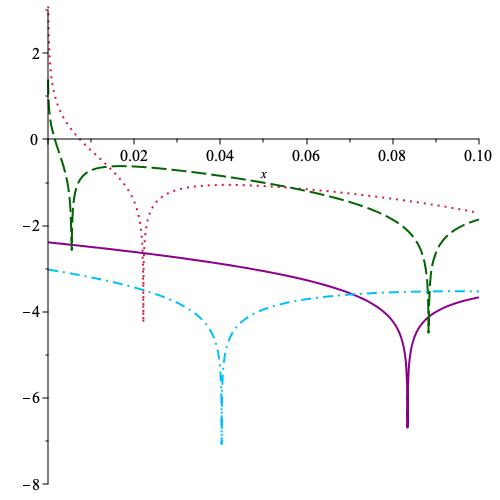}
    \caption{Laguerre W-functions: A close-up of the bottom plot in Fig.~\ref{fig:4.5} near the left endpoint.}
    \label{fig:4.6}
  \end{center}
\end{figure}

\subsection{Computational and theoretical challenges}

This is the first paper to consider W-functions in an organised way, although of course Hermite functions have been used and investigated extensively and W-functions associated with Freud weights (and which are special because of Theorem~2) have been introduced in \cite{luong23apw}. Needless to say, this work neither resolves all the mathematical and computational issues associated with W-functions nor claims to do so. While there are important theoretical questions, e.g.\ to characterise all separable or symmetrically separable weight functions, perhaps the most urgent issues are related to the applications of W-functions to spectral methods. This concerns issues in approximation theory (speed of convergence in different function classes), as well as purely computational questions. The speed of approximation points out to an imperfect duality between  W-functions and the functions $\Psi=\{\psi_n\}_{n\in\bb{Z}_+}$ from Section~1. Recalling the $\hat{f}_n^\Phi$, the $n$th expansion coefficient in $\PPP$ and letting $\check{w}(x)=w(x)\chi_{(a,b)}(x)$, the Plancherel theorem yields at once for every $n\in\BB{Z}_+$
\begin{displaymath}
  \hat{f}_n^\Phi=\int_a^b f(x) \phi_n(x)\D x=\int_{-\infty}^\infty \sqrt{\check{w}(x)} f(x) p_n(x)\D x=(-\ii)^n \int_{-\infty}^\infty \hat{f}(\xi) \overline{\psi_n(x)}\D x,
\end{displaymath}
and we recover an expansion in $\Psi$ of the Fourier transform of $f$. This duality, though, is imperfect because, unless $(a,b)=\BB{R}$, it is valid (insofar as $\Psi$ is concerned) only in the Paley--Wiener space $\mathcal{P}_{(a,b)}(\BB{R})$ rather than in $\CC{L}_2(\BB{R})$ \cite{iserles19oss}. Moreover, comprehensive convergence theory for functions of the form $\Psi$ is also lacking. Yet, even an imperfect duality might potentially lead to useful observations.

The final issue we wish to mention is fast computation. For example, while Subsection~4.1 provides a viable approach toward the computation of $g(\DDD_N)\MM{v}$ using Krylov subspaces, ideally it would have been useful to have other fast methods, in particular for $g(z)=\ee^z$. Another critical issue is rapid expansion in a W-function basis, similarly perhaps to  fast expansion algorithms in polynomial bases \cite{olver20fau}. All this is a matter for future research.

\bibliographystyle{agsm}

\appendix
\def\theequation{\Alph{section}.\arabic{equation}}
\section{Separability coefficients for Laguerre weights}

Our starting point is the generating function
\begin{displaymath}
  \sum_{n=0}^\infty \LL_m^{(\alpha)}(x)z^n=\frac{1}{(1-z)^{1+\alpha}} \exp\!\left(\frac{xz}{z-1}\right)
\end{displaymath}
\cite[p.~202]{rainville60sf}, and we recall that 
\begin{equation}
  \label{eq:A.1}
  \int_0^\infty [\LL_m^{(\alpha)}(x)]^2x^\alpha \ee^{-x}\D x=\frac{\G(m+1+\alpha)}{m!}=\frac{\G(1+\alpha)(1+\alpha)_m}{m!}.
\end{equation}

Set
\begin{displaymath}
  q_{m,n}=-\frac12\int_0^\infty \LL_m^{(\alpha)}(x)\LL_n^{(\alpha)}(x) \frac{\D x^\alpha\ee^{-x}}{\D x}\D x,\qquad m,n\in\BB{Z}_+.
\end{displaymath}
 Because of 
 \begin{displaymath}
  \tilde{p}_m(x)=\sqrt{\frac{m!}{\G(m+1+\alpha)}} \LL_m^{(\alpha)}(x)
\end{displaymath}
and \R{eq:2.3}, it follows that
\begin{displaymath}
  \DDD_{m,n}= \sqrt{\frac{m!n!}{\G(m+1+\alpha)\G(n+1+\alpha)}} q_{m,n},\qquad m\geq n+1.
\end{displaymath}

We set
\begin{Eqnarray*}
  Q(s,t)&=&\sum_{m=0}^\infty \sum_{n=0}^\infty q_{m,n} s^m t^n\\
  &=&-\frac12 \int_0^\infty \sum_{m=0}^\infty \sum_{n=0}^\infty \LL_m^{(\alpha)}(x) \LL_n^{(\alpha)}(x) s^m t^n  (-x^\alpha +\alpha x^{\alpha-1})\ee^{-x}\D x. 
\end{Eqnarray*}
Using \R{eq:A.1} it follows that  for $|s|,|t|<1$,
\begin{Eqnarray*}
  Q(s,t)&=&\frac12 \sum_{m=0}^\infty \frac{\G(m+1+\alpha)}{m!} (st)^m \\
  &&\mbox{}-\frac{\alpha}{2(1-s)^{1+\alpha}(1-t)^{1+\alpha}} \int_0^\infty x^{\alpha-1} \exp\!\left(-x+\frac{xt}{t-1}+\frac{xs}{s-1}\right)\!\D x\\
  &=&\frac12 \G(1+\alpha) \hyper{1}{0}{1+\alpha}{\mbox{---}}{st}\\
  &&\mbox{}-\frac{\alpha}{2(1-s)^{1+\alpha}(1-t)^{1+\alpha)}} \int_0^\infty x^{\alpha-1} \exp\!\left(-\frac{(1-ts)x}{(1-s)(1-t)}\right)\!\D x\\
  &=&\frac{\G(1+\alpha)}{2} \frac{1}{(1-st)^{\alpha+1}}-\frac{\G(1+\alpha)}{2(1-s)(1-t)(1-st)^\alpha}.
\end{Eqnarray*}
(Cf. for example \cite{rainville60sf} for the definition and basic facts on hypergeometric functions.) We now  expand: all it takes is elementary (but long) algebra:
\begin{Eqnarray*}
  \frac{1}{(1-st)^{1+\alpha}}&=&\sum_{m=0}^\infty \frac{(1+\alpha)_m}{m!} (st)^m,\\
  \frac{1}{(1-s)(1-t)(1-st)^\alpha}&=&\sum_{m=0}^\infty \sum_{n=0}^\infty \left[\sum_{k=0}^{\min\{m,n\}} \frac{(\alpha)_k}{k!}\right] s^m t^n.
\end{Eqnarray*}
The following proposition can be trivially proved by induction.\footnote{And it might well be already known.}

\begin{proposition}
  \begin{displaymath}
    \sum_{k=0}^j \frac{(\alpha)_k}{k!} =\frac{(1+\alpha)_j}{j!}.
  \end{displaymath}
\end{proposition}

We thus deduce from the definition of $Q$ that $q_{n,n}=0$ and
\begin{displaymath}
  q_{m,n}=\frac{\G(\alpha)}{2} \frac{(1+\alpha)_{n}}{n!},\qquad m\geq n+1,
\end{displaymath}
and conclude that 
\begin{equation}
  \label{eq:A.2}
  \DDD_{m,n}=\frac{\G(1+\alpha)}{2} \sqrt{\frac{m!n!}{\G(m+1+\alpha)\G(n+1+\alpha)}}  \frac{(1+\alpha)_{n}}{n!}=\frac12 \sqrt{\frac{m!\G(n+1+\alpha)}{\G(m+1+\alpha)n!}}
\end{equation}
for $m\geq n+1$, with skew-symmetric completion for $m\leq n$.

We have just determined both separability and the explicit form of the sequences $\GG{a}$ and $\GG{b}$.

\begin{theorem}
  \label{th:Laguerre}
  The Laguerre weight is separable and
  \begin{equation}
    \label{eq:A.3}
    \GG{a}_m=\sqrt{\frac{m!}{2\G(m+1+\alpha)}},\qquad \GG{b}_n=\sqrt{\frac{\G(n+1+\alpha)}{2n!}},\qquad m,n\in\BB{Z}_+.
  \end{equation}
\end{theorem}

\setcounter{equation}{0}
\section{Symmetric separability coefficients for ultraspherical weights}

We recall that
\begin{displaymath}
  S_{m,n}^\alpha=\int_{-1}^1 (1-x^2)^{\alpha-1} \PP_m^{(\alpha,\alpha)}(x)\PP_n^{(\alpha,\alpha)}(x)\D x,\qquad \alpha>0,
\end{displaymath}
and we are concerned with $m\geq n$ ($S_{m,n}$ is symmetric) and even $m+n$. We commence by dividing $\PP_n^{(\alpha,\alpha)}$ by $(1-x^2)$ -- it follows from the Euclidean algorithm that 
\begin{displaymath}
  \PP_n^{(\alpha,\alpha)}(x)=(1-x^2) \gamma_n(x)+\delta_n(x),
\end{displaymath}
where $\gamma_n\in\BB{P}_{n-1}$ and $\delta_n(x)=\delta_{n,0}+\delta_{n,1} x$ is linear. Because of parity, if $n$ is even then $\delta_1=0$, while if it is odd then $\delta_0=0$. The description of $\delta$ can be completed by considering $x=1$,
\begin{displaymath}
  \delta_{2n,0}=\PP_{2n}^{(\alpha,\alpha)}(1)=\frac{(1+\alpha)_{2n}}{(2n)!},\qquad \delta_{2n+1,1}=\PP_{2n+1}^{(\alpha,\alpha)}(1)=\frac{(1+\alpha)_{2n+1}}{(2n+1)!},
\end{displaymath}
therefore
\begin{displaymath}
  \delta_{2n}\equiv \frac{(1+\alpha)_{2n}}{(2n)!},\qquad \delta_{2n+1}(x)=\frac{(1+\alpha)_{2n+1}}{(2n+1)!}x.
\end{displaymath}
Since $\deg\gamma\leq n-1\leq m-1$, it follows from orthogonality that
\begin{Eqnarray*}
  S_{m,n}^\alpha&=&\int_{-1}^1 (1-x^2)^{\alpha-1} \PP_m^{(\alpha,\alpha)}(x)[(1-x^2)\gamma_n(x)+\delta_n(x)]\D x\\
  &=&\int_{-1}^1 (1-x^2)^{\alpha-1} \PP_m^{(\alpha,\alpha)}(x) \delta_n(x)\D x.
\end{Eqnarray*}
Letting
\begin{displaymath}
  \GG{e}_m=\int_{-1}^1 (1-x^2)^{\alpha-1} \PP_{2m}^{(\alpha,\alpha)}(x)\D x,\qquad \GG{o}_m=\int_{-1}^1 (1-x^2)^{\alpha-1} x \PP_{2m+1}^{(\alpha,\alpha)}(x)\D x,
\end{displaymath}
we thus have
\begin{displaymath}
  S_{2m,2n}^\alpha=\frac{(1+\alpha)_{2n}}{(2n)!} \GG{e}_m,\qquad S_{2m+1,2n+1}^\alpha=\frac{(1+\alpha)_{2n+1}}{(2n+1)!} \GG{o}_m.
\end{displaymath}

We wish to prove that
\begin{equation}
  \label{eq:B.1}
  \GG{e}_m=\frac{4^\alpha}{\alpha}\frac{\G(2m+1+\alpha)\G(1+\alpha)}{\G(2m+1+2\alpha)},\qquad \GG{o}_m=\frac{4^\alpha}{\alpha} \frac{\G(2m+2+\alpha)\G(1+\alpha)}{\G(2m+2+2\alpha)}.
\end{equation}
To this end, it is helpful to rewrite \R{eq:B.1} in the form 
\begin{equation}
  \label{eq:B.2}
  \GG{e}_m=\frac{\sqrt{\pi}}{4^m}\frac{(m+1+\alpha)_m\G(\alpha)}{\G(m+\alpha+\frac12)},\qquad \GG{o}_m=\frac{\sqrt{\pi}(m+1+\alpha)_{m+1}\G(\alpha)}{2\cdot 4^m\G(\alpha+m+\frac32)}.
\end{equation}
To prove that \R{eq:B.1} is identical to \R{eq:B.2} for $\GG{e}_m$ we commence from the latter, noting that it is the same as
\begin{displaymath}
  \GG{e}_m=\frac{\sqrt{\pi}}{4^m}\frac{\G(2m+1+\alpha)\G(\alpha)}{\G(m+\alpha+\frac12)\G(m+1+\alpha)}
\end{displaymath}
and use the Gamma duplication formula 
\begin{displaymath}
  \G(2z)=\pi^{-1/2} 2^{2z-1} \G(z)\G(z+\Frac12),\qquad z\in\BB{C}\setminus-\BB{Z}_+
\end{displaymath}
\cite[5.5.5]{dlmf}. Letting $z=m+\alpha+\frac12$, we have
\begin{displaymath}
  \G(m+\alpha+\Frac12)\G(m+1+\alpha)=\frac{\sqrt{\pi}\G(2m+1+2\alpha)}{4^{m+\alpha}}.
\end{displaymath}
and obtain \R{eq:B.1} following elementary manipulation. An identical procedure applies to $\GG{o}_m$.

Replacing $m$ by $2m$ in \R{eq:3.6} results in the recursion
\begin{equation}
  \label{eq:B.3}
  \GG{e}_{m}=\frac12 \frac{(2m+\alpha)(4m-1+2\alpha)}{m+\alpha} \GG{o}_{m-1}-\frac12 \frac{(2m-1+\alpha)(2m+\alpha)}{m+\alpha} \GG{e}_{m-1},
\end{equation}
while replacing $m$ by $2m+1$ results in 
\begin{Eqnarray*}
  \GG{o}_m&=&\frac{(2m+1+\alpha)(4m+1+2\alpha)}{2m+1+2\alpha} \int_{-1}^1 (1-x^2)^{\alpha-1} x^2\PP_{2m}^{(\alpha,\alpha)}(x)\D x\\
  &&\mbox{}-\frac{(2m+\alpha)(2m+1+\alpha)}{2m+1+2\alpha} \GG{o}_{m-1}.
\end{Eqnarray*}
Replacing $x^2=1-(1-x^2)$ and using orthogonality,
\begin{Eqnarray*}
  \int_{-1}^1 (1-x^2)^{\alpha-1} x^2\PP_{2m}^{(\alpha,\alpha)}(x)\D x&=&\int_{-1}^1 (1-x^2)^{\alpha-1}\PP_{2m}^{(\alpha,\alpha)}(x)\D x\\
  &&\mbox{}-\int_{-1}^1 (1-x^2)^\alpha \PP_{2m}^{(\alpha,\alpha)}(x)\D x=\GG{e}_m
\end{Eqnarray*}
for $m\in\BB{N}$.  Thus,
\begin{equation}
  \label{eq:B.4}
  \GG{o}_m=\frac{(2m+1+\alpha)(4m+1+2\alpha)}{2m+1+2\alpha}\GG{e}_m-\frac{(2m+\alpha)(2m+1+\alpha)}{2m+1+2\alpha}\GG{o}_{m-1}.
\end{equation}

We compute directly
\begin{displaymath}
  \GG{e}_0=\frac{\sqrt{\pi}(1+\alpha)\G(\alpha)}{\G(\alpha+\frac12)},\qquad \GG{o}_0=\frac{\sqrt{\pi}(1+\alpha)\G(\alpha)}{2\G(\alpha+\frac32)}
\end{displaymath}
(this is consistent with \R{eq:B.2} for $m=0$), whereby \R{eq:B.2} follows from \R{eq:B.3} and \R{eq:B.4} by easy induction.  We deduce that
\begin{equation}
  \label{eq:B.5}
  S_{m,n}^\alpha=\frac{4^\alpha}{\alpha} \frac{\G(m+1+\alpha)\G(n+1+\alpha)}{n!\G(m+1+2\alpha)},\qquad m\geq n,\quad m+n\mbox{ even}.
\end{equation}

\end{document}